\documentclass[11pt,a4paper,reqno]{amsart}

\usepackage[T1]{fontenc}
\usepackage[utf8]{inputenc}
\usepackage{lmodern}
\usepackage{microtype}
\usepackage[
  a4paper,
  textwidth=15.2cm,
  textheight=23.2cm,
  centering,
  headheight=14pt
]{geometry}
\usepackage{mathrsfs}
\usepackage{amsmath,amssymb,amsthm}
\usepackage{enumerate}

\usepackage[
  colorlinks=true,
  linkcolor=blue,
  citecolor=magenta,
  urlcolor=blue
]{hyperref}

\def\XXint#1#2#3{{\setbox0=\hbox{$#1{#2#3}{\int}$ }
\vcenter{\hbox{$#2#3$ }}\kern-.6\wd0}}

\newtheorem{theorem}{\bf Theorem}[section]
\newtheorem{proposition}[theorem]{\bf Proposition}
\newtheorem{lemma}[theorem]{\bf Lemma}
\newtheorem{corollary}[theorem]{\bf Corollary}

\theoremstyle{definition}
\newtheorem{definition}[theorem]{Definition}

\newtheorem{remark}[theorem]{Remark}
\newtheorem{hypothesis}[theorem]{Hypothesis}

\numberwithin{equation}{section}

\begin{document}

\title[Positive solutions: sublinear equations]{Positive Solutions for 
Sublinear Equations with Compact Positivity-Improving Resolvent}

\author[T. Klimsiak]{Tomasz Klimsiak}

\thanks{This work was supported by the Polish National Science Centre
(Grant No.~2022/45/B/ST1/01095).}

\address{
Faculty of Mathematics and Computer Science,
Nicolaus Copernicus University in Toru\'n,
Chopina 12/18,
87-100 Toru\'n,
Poland
}

\email{tomas@mat.umk.pl}



\begin{abstract}
We establish a Br\'ezis--Oswald-type spectral principle for semilinear
equations
\[
-Lu=f(x,u),
\]
where $L$ is the generator of a positive $C_0$-semigroup on $L^p$ with
compact positivity-improving resolvent. Neither symmetry, variational
structure, nor regularizing properties such as ultracontractivity or
smoothing are assumed. Let $a_0$ and $a_\infty$ denote the asymptotic
slopes of the nonlinearity at zero and at infinity, and let
$\lambda_1(a)$ denote the generalized principal eigenvalue associated
with the perturbed operator $-L-a$. We prove that
\[
\lambda_1(a_0)<0<\lambda_1(a_\infty)
\]
implies that the equation admits a strictly positive solution.

The proof develops a potential-theoretic sub- and supersolution framework
based on the order induced by supermedian functions and combines a
Deny-type compactness theorem, a Kato-type inequality, and a Doob
transform. The construction also yields an order-preserving selection of
solutions and allows the lower control on the nonlinearity to be relaxed.
If $y\mapsto f(x,y)/y$ is strictly decreasing and $a_0$ is bounded, the
spectral condition is necessary as well, and the strictly positive
solution is unique.

\end{abstract}

\maketitle

\section{Introduction}
\label{sec1}

The existence and uniqueness of positive solutions are classical questions in the theory
of semilinear elliptic equations. A fundamental spectral criterion for
sublinear problems was established by Br\'ezis and Oswald
\cite{BrezisOswald1986}. For the Dirichlet problem
\[
-\Delta u=f(x,u)\quad\text{in }D,
\qquad
u=0\quad\text{on }\partial D,
\]
under suitable one-sided growth assumptions on $f$, they showed that
the existence of a positive solution is governed by spectral inequalities determined by the
behavior of the quotient $f(x,y)/y$ at zero and at infinity. If, in addition,
this quotient is strictly decreasing in $y>0$, the corresponding spectral condition
becomes necessary and the positive solution is unique.

The purpose of the present paper is to identify an operator-theoretic
mechanism behind the Br\'ezis--Oswald spectral principle. 
We show that the existence part of this principle persists within the abstract framework of positive $C_0$-semigroups
on $L^p$ with compact positivity-improving resolvent. In particular,
the argument requires neither symmetry nor a variational structure and
does not rely on smoothing or ultracontractive properties of the semigroup,
on any special differential structure of the operator, or on geometric
properties of the domain.

More precisely, let $(E,\mathcal A)$ be a Lusin measurable space endowed
with a finite measure $m$, and let $L$ generate a strongly continuous
positive semigroup $(T_t)$ on $L^p(E;m)$, $1\le p<\infty$, whose resolvent is
compact and positivity improving. In view of the invariance of our assumptions under spectral shifts,
we may and do assume that $s(L)<0$. We study strictly positive
solutions of
\begin{equation}\label{eq1.1}
    -Lu=f(x,u)\quad\text{in }E,
\end{equation}
for Carath\'eodory nonlinearities satisfying a one-sided linear upper
growth condition and a local lower control near zero. The precise
hypotheses are stated as \textup{(F1)--(F3)} in
Section~\ref{sec.prelim}.

\subsection{Main results}
The spectral quantities governing the existence problem are determined by
the asymptotic slopes
\[
a_0(x)
:=
\liminf_{y\to0^+}\frac{f(x,y)}{y},
\qquad
a_\infty(x)
:=
\limsup_{y\to\infty}\frac{f(x,y)}{y}.
\]
Neither $a_0$ nor $a_\infty$ is assumed to be bounded. Accordingly, the
corresponding generalized principal eigenvalues are defined by monotone
truncation:
\[
\lambda_1(a_0)
:=
\inf_{k\geq1}
\lambda_1\bigl(-L-(a_0\wedge k)\bigr),
\qquad 
\lambda_1(a_\infty)
:=
\sup_{k\geq1}
\lambda_1\bigl(-L-(a_\infty\vee(-k))\bigr).
\]
The notion of solution, which requires a slight extension of the usual
resolvent formulation, is specified below.

\medskip
\noindent
\textbf{Main theorem.}
Under the above assumptions, if
\begin{equation}\label{eq:spectral}
    \lambda_1(a_0)<0<\lambda_1(a_\infty),
\end{equation}
then \eqref{eq1.1} admits a strictly positive solution.
If, in addition, $y\mapsto f(x,y)/y$ is  non-increasing for
$m$-a.e. $x\in E$ and decreasing on a set of positive measure,  and $a_0$ is bounded, then \eqref{eq:spectral}
is also necessary for existence and the strictly positive solution is
unique.

A distinctive feature of the present setting is that solutions need not
be bounded and, consequently, $f(\cdot,u)$ need not belong to
$L^p(E;m)$. Indeed, while the upper growth condition ensures that
$f^+(\cdot,u)\in L^p(E;m)$ for $u\in L^p(E;m)$, the negative part
$f^-(\cdot,u)$ need not belong to $L^p(E;m)$. Thus the equation
cannot in general be written in the standard resolvent form
\[
u=Rf(\cdot,u),
\qquad R=(-L)^{-1}.
\]
We therefore separate the positive and negative parts of $f$. A function
$u\in L^p(E;m)$ is called a solution of \eqref{eq1.1} if
$Rf^-(\cdot,u)\in L^p(E;m)$ and
\[
u+Rf^-(\cdot,u)=Rf^+(\cdot,u).
\]
By Rohlin's theorem, the positive resolvent $R$ admits a kernel
representation and hence extends naturally from $L^p(E;m)$ to
nonnegative measurable functions. Whenever
$f^-(\cdot,u)\in L^p(E;m)$, this notion agrees with the usual resolvent
formulation of the equation.

The method developed in the paper also yields an order-preserving selection principle. 
Within a fixed class of nonlinearities satisfying the assumptions of the existence theorem, a strictly positive solution
$u_f$ can be selected for each $f$ so that
\[
f_1\le f_2\quad\Longrightarrow\quad u_{f_1}\le u_{f_2}.
\]
The selection
principle is then used to weaken the local lower-control assumption in the
main existence theorem; the precise formulation is given in
Section~\ref{sec7}.

\subsection{Related literature}
The Br\'ezis--Oswald spectral criterion has subsequently been extended in
several directions. For nonsymmetric uniformly elliptic operators in
nondivergence form on possibly irregular bounded domains, we refer to Bao
\cite{Bao2003}. Quasilinear extensions include the $p$-Laplacian and more
general nonhomogeneous operators subject to different boundary conditions;
see, for instance,
\cite{DiazSaa1987,FragnelliMugnaiPapageorgiou2016}. Nonlocal analogues have
been obtained for the fractional $p$-Laplacian with Robin and Dirichlet
conditions
\cite{MugnaiPinamontiVecchi2020,IannizzottoMugnai2024}, as well as for mixed
local--nonlocal operators \cite{BiagiMugnaiVecchi2024}. These results
typically exploit structures specific to the operator under consideration,
such as a variational formulation, a Picone-type identity, a comparison
principle, boundary regularity, or detailed information on the corresponding
weighted eigenvalue problems.

Among abstract operator-theoretic approaches, the work of Arendt and Daners \cite{AD} 
is particularly close in spirit to the present paper. Working within the theory of positive 
semigroups and Banach lattices, they studied the abstract degenerate logistic equation
\begin{equation}
\label{eq.deglog1}
-Lu=\lambda u-b(x)g(x,u)u \quad\text{in }D,
\end{equation}
on possibly rough domains $D\subset\mathbb R^d$, where $b$ is a bounded, positive, nontrivial function on $D$, and
$g\in C^{0,1}(\overline D\times\mathbb R)$ is positive, strictly increasing in the second variable, and satisfies
\[
\lim_{y\to\infty}g(x,y)=\infty
\]
uniformly in $x$ on compact subsets of $D$. Their abstract framework makes it possible, in particular, 
to dispense with boundary regularity assumptions inherent in earlier treatments of problems of the form \eqref{eq.deglog1}.
The present setting differs from that of \cite{AD} in two essential respects. We allow a general one-sided 
sublinear nonlinearity, and we impose neither the local smoothing nor the ultracontractive 
properties of the semigroup that enter essentially into their analysis of the degenerate logistic equation.
Thus the issue is not merely to treat another particular class of
elliptic or nonlocal operators, but to determine whether the
Br\'ezis--Oswald spectral mechanism can be recovered from compactness
and positivity improvement of the resolvent alone, without
operator-specific identities or regularizing estimates. In particular, the existence argument applies without monotonicity of
$f(x,y)/y$, allows unbounded solutions and asymptotic slopes, and does not
derive compactness of the nonlinear problem from any regularizing estimate.

\subsection{Proof strategy}
To work at the level of the resolvent rather than the generator domain, 
we develop a sub- and supersolution
framework based on supermedian functions, that is, nonnegative functions $\gamma\in L^p(E;m)$ satisfying
\[
\alpha R_\alpha\gamma\leq\gamma,
\qquad \alpha>0.
\]
The construction is naturally formulated in terms of the order induced by the cone of supermedian functions: 
for $h,g\in L^p(E;m)$, we write $h\preceq g$ whenever $g-h$ is supermedian. 
This is in the spirit of order structures arising in abstract potential theory; see, for instance, \cite{BB}. 
In the usual abstract sub- and supersolution theory on Banach lattices, a subsolution to $-Lu=f(\cdot,u)$ 
is typically required to belong to $\mathscr D(L)$ and to satisfy $-Lw\leq f(\cdot,w)$. Here this domain-based 
condition is replaced, at the resolvent level, by
\[
w+Rf^-(\cdot,w)\preceq Rf^+(\cdot,w),
\]
with the analogous reversed relation for supersolutions. 
The main point is that this order is compatible with the lattice
operations needed in the nonlinear argument; in particular, the maximum of two positive
subsolutions is again a subsolution.
The resulting flexibility is essential in the construction of suitable barriers for a general one-sided sublinear nonlinearity.

The proof is first carried out under the additional assumption that the
semigroup $(T_t)$ is sub-Markovian. Two ingredients are developed at this
stage. 
First, we prove a Deny-type compactness theorem for supermedian
functions. This provides the compactness needed in the nonlinear
argument without appealing to smoothing properties of the semigroup.
Thus compactness enters the proof through potential theory rather
than through regularization of the semigroup.
Second, we establish a Kato-type inequality for
sub-Markovian semigroups with compact positivity-improving resolvent. These
tools allow us to develop a method of sub- and supersolutions for nonlinear
resolvent equations of the form
\begin{equation}\label{eq:abstract-fixed-point}
u=\gamma+Rg(\cdot,u),
\end{equation}
where $\gamma$ is supermedian. In this way, the analysis is carried out in
a class of functions extending beyond the domain of the operator $L$.
The supermedian term $\gamma$ may be viewed as an abstract analogue of
boundary data. Allowing such a term is also essential in the proof of our main existence theorem, 
where we first solve suitable perturbed equations with strictly positive supermedian 
data and then pass to the limit as this perturbation vanishes.

The probabilistic realization of the resolvent is essential in establishing
these tools. By a theorem of Beznea, Boboc, and R\"ockner \cite{BBR}, the
$L^p$-resolvent can be represented, after passage to a larger Lusin space,
as the resolvent of a right Markov process. This makes available
probabilistic arguments for excessive and supermedian functions while
preserving the original $L^p(E;m)$ formulation of the problem.

Finally, the sub-Markovian assumption is removed by a Doob transform
associated with a supermedian function for  $(T_t)$. The transform reduces
the general positive-semigroup case to the sub-Markovian one while preserving
the spectral quantities relevant to \eqref{eq:spectral}, and hence extends
the existence result to the full class of operators considered in the paper.

\subsection{Organization of the paper}

The paper is organized as follows. Section~\ref{sec.prelim} contains the
analytic preliminaries and the spectral theory of positive compact
resolvents and their perturbations. Section~\ref{sec2} recalls the necessary
facts concerning right Markov processes and the realization theorem used
throughout the probabilistic part of the paper. In Section~\ref{sec3}, we
establish a Deny-type compactness theorem for supermedian functions, while
Section~\ref{sec4} is devoted to a Kato-type inequality. Section~\ref{sec5}
develops the method of sub- and supersolutions for nonlinear resolvent
equations. In Section~\ref{sec6}, we prove the main existence theorem in the
sub-Markovian case. 
Section~\ref{sec7} removes the sub-Markovian assumption by means of the
Doob transform, establishes the general existence theorem and the
monotone selection property described above, and finally uses this
selection to derive the existence result under the weaker lower control condition
on the nonlinearity.
The final section contains the necessary spectral condition under the strict
monotonicity assumption and the uniqueness theorem.

\section{Preliminaries}
\label{sec.prelim}

Throughout the paper, we work on a Lusin measurable space $(E,\mathcal A)$ equipped with a finite measure $m$. 
Let us recall that  $(E,\mathcal{A})$ is called \emph{Lusin} if it is measurably isomorphic to a Borel subset of a Polish space; 
that is, if there exist a Polish space $X$, a Borel set $B \subset X$, and a bijection
\[
\phi : E \to B
\]
such that both $\phi$ and $\phi^{-1}$ are measurable.
We fix  $p\in [1,\infty)$ and denote  by $L^p(E;m)$ the space of all $\mathcal A$-measurable real functions 
$g$ on $E$ such that $|g|^p$ is integrable, modulo the usual $m$-a.e. equivalence relation.
For a measurable function $u:E\to\mathbb R$, $u>0$ means $m(\{u>0\})>0$ and $m(\{u<0\})=0$, 
while $u\gg 0$ means $m(\{u\le 0\})=0$. 
Recall that a positive $C_0$-semigroup $(S_t)$ on  $L^p(E;m)$
is called irreducible if $\{0\}$ and $L^p(E;m)$ are the only closed ideals 
that are invariant under all the operators $S_t, t \geq 0$.   Let $(S_t)$ be a positive $C_0$-semigroup  on  $L^p(E;m)$
with generator $(A,\mathscr D(A))$ and resolvent $(G_\alpha)$. By $s(A)$ we denote the spectral bound of $A$, i.e. 
\[
s(A):=\sup \{\operatorname{Re} \lambda: \lambda \in \sigma(A)\}.
\]
We say that $(G_\alpha)$ is {\em positivity improving} if, whenever  $f>0$, 
$f\in L^p(E;m)$, and $\alpha>s(A)$ one has  $G_\alpha f\gg0$.
By \cite[Proposition 14.10]{BatkaiKramarFijavzRhandi2017} a positive $C_0$-semigroup $(S_t)$ on 
$L^p(E;m)$ with a generator $(A,\mathscr D(A))$ is irreducible if and only if its resolvent $(G_\alpha)$ is positivity improving. 

Throughout the paper, we follow the terminology customary in potential theory, 
according to which a positive function means a nonnegative function.

In the following proposition we collect some  known results from  spectral theory that will be used frequently below.

\begin{proposition}
\label{prop.tool}
Let $(S_t)$ be a positive  $C_0$-semigroup on $L^p(E;m)$ that generates an operator $(A,\mathscr D(A))$
with compact  positivity-improving resolvent. 
Then  the following hold:
\begin{itemize}
\item[(i)] $\sigma(A) \neq \emptyset$, hence $s(A)>-\infty$.
\item[(ii)] The number $\lambda_1(-A):=-s(A)$ is an eigenvalue of $-A$.
\item[(iii)] There exists $u \in \mathscr D(A)$ such that $u\gg0$ and  $-A u=\lambda_1(-A) u$.
\item[(iv)] The geometric multiplicity of the eigenvalue $\lambda_1(-A)$ is one.
\item[(v)] There exists  $f\in \mathscr D(A^*)$ such that $f\gg0$ and $-A^*f=\lambda_1(-A) f$.
\item[(vi)] If $u\in \mathscr D(A)$ is positive and nontrivial, and $-Au=\lambda u$ for some $\lambda\in\mathbb R$,
then $\lambda=\lambda_1(-A)$. 
\end{itemize}
\end{proposition}
\begin{proof}
The assertion (i) is a consequence of  a result of de Pagter \cite[Theorem 3]{dePagter1986}.
For (ii) see \cite[Corollary 12.9]{BatkaiKramarFijavzRhandi2017}.
The assertion (iii) follows from the Krein-Rutman theorem (see, e.g., \cite[Theorem 12.15]{BatkaiKramarFijavzRhandi2017})
and the positivity-improving property of the resolvent.  The assertions (iv) and (v) follow from 
 \cite[Proposition 14.12(d)]{BatkaiKramarFijavzRhandi2017} (note that by \cite[Chapter~IV, Corollary~1.19]{EngelNagel2000}, 
 the spectrum of an operator with compact resolvent consists only of poles of the resolvent of finite algebraic multiplicity). 
 For (vi) see, e.g.,  \cite[Theorem V.5.2(iv)]{Schaefer1974}.
\end{proof}

In what follows, any pair  $(\lambda_1(-A),u)$ as in  Proposition 
\ref{prop.tool} will be  called a principal  eigenpair  associated with   $(A,\mathscr D(A))$; 
here $\lambda_1(-A)$ is  the principal eigenvalue and $u$ is a principal eigenfunction. Similarly, any  pair  
$(\lambda_1(-A),f)$ as in  Proposition \ref{prop.tool} will be called a dual 
principal eigenpair associated with   $(A,\mathscr D(A))$, and
 $f$ will be  called a dual  principal eigenfunction.

\begin{corollary}
\label{cor.pert1}
Let $(A,\mathscr D(A))$ generate a positive $C_0$-semigroup $(S_t)$ on $L^p(E;m)$.
Assume that $A$ has compact and positivity-improving resolvent and  $a\in L^\infty(E;m)$.
Then the operator $A-a$
generates a positive $C_0$-semigroup $(S_t^a)$ on $L^p(E;m)$ that 
has compact positivity-improving resolvent.
Finally, if $a_1,a_2\in L^\infty(E;m)$  and
$a_1\le a_2$ $m$-a.e., then
\[
\lambda_1(-A+a_1)\le \lambda_1(-A+a_2).
\]
\end{corollary}

\begin{proof}
Denote by $M_a$ the bounded
multiplication operator $M_af=af$. By the bounded perturbation theorem,
$A-M_a$ generates a $C_0$-semigroup on $L^p(E;m)$; see
\cite[Theorem III.1.3]{EngelNagel2000}.
Since the semigroup generated by $-M_a$ is the positive multiplication semigroup
$f\mapsto e^{-ta}f$, positivity of $(S_t^a)_{t\ge0}$ follows from the Trotter product
formula, see \cite[Corollary III.5.8]{EngelNagel2000}.
To prove irreducibility of $(S_t^a)_{t\ge0}$, note that
\[
A-a=(A-\|a\|_\infty )+M_{\|a\|_\infty-a},
\qquad M_{\|a\|_\infty-a}\ge0.
\]
Hence, by \cite[Corollary VI.1.11(i)]{EngelNagel2000},
\[
0\le e^{-\|a\|_\infty t}S_t\le S_t^a,
\qquad t\ge0.
\]
Therefore every closed ideal invariant under $(S_t^a)$ is also invariant under
$(S_t)$, and since $(S_t)$ is irreducible, so is $(S_t^a)$.
Compactness of the resolvent of $A-a$ follows from the  resolvent identity
\[
R(\lambda,A-a)=R(\lambda,A)\bigl(I+M_aR(\lambda,A)\bigr)^{-1},
\qquad \lambda\in\rho(A)\cap\rho(A-a).
\]
Since $R(\lambda,A)$ is compact and $\bigl(I+M_aR(\lambda,A)\bigr)^{-1}$ is bounded,
$R(\lambda,A-a)$ is compact.
If now $a_1\le a_2$ $m$-a.e., then
\[
A-a_1=(A-a_2)+M_{a_2-a_1},
\qquad M_{a_2-a_1}\ge0.
\]
By monotonicity of the spectral bound under positive bounded perturbations,
\cite[Corollary VI.1.11(ii)]{EngelNagel2000},
\[
s(A-a_2)\le s(A-a_1),
\]
which, combined with Proposition \ref{prop.tool}(ii),   completes the proof.
\end{proof}

\begin{hypothesis}
\label{hyp.spectral-bound}
Throughout the remainder of the paper, unless stated otherwise,
$(L,\mathscr D(L))$ denotes a closed operator on $L^p(E;m)$ generating a positive $C_0$-semigroup $(T_t)$
whose resolvent $(R_\alpha)$  is compact and positivity improving.
Furthermore,  the spectral bound of $L$ is
assumed to be  strictly negative, that is,
\[
s(L)<0.
\]
Consequently, $0\in\rho(L)$, and throughout the paper we write
\[
R:=R_0=(-L)^{-1}.
\]
\end{hypothesis}

For a given strictly positive supermedian function $\psi$, 
we consider the following condition:
\begin{enumerate}
\item[(F3)$_\psi$]   for any $\delta>0$ there exists $C_\delta\ge 0$  such that $f(x,y)\ge -C_\delta y$ 
for $y\in [0,\delta \psi(x)]$ and $x\in E$.
\end{enumerate}

\begin{hypothesis}
\label{hyp.fff}
We  assume throughout that
$f:  E\times \mathbb R\to \mathbb R$ is a Carath\'eodory function satisfying the following conditions:

\begin{enumerate}
\item[(F1)] there exists $\lambda\ge 0$ and positive $h\in L^p(E;m)$ such that $f(x,y)\le h(x)+\lambda y,\, y\ge 0$;
\item[(F2)] $[0,\infty)\ni y\mapsto f(x,y)$ is continuous for any $x\in E$;
\item[(F3)]  there exists a strictly positive supermedian function $\psi$ such that (F3)$_\psi$ holds.
\end{enumerate}
\end{hypothesis}

We introduce, for any function $g:E\times\mathbb R\to\mathbb R$,  the generalized slopes
\begin{equation}
\label{eq.slopes}
a_0(g)(x)=\liminf_{y\to0^+}\frac{g(x,y)}{y},\qquad
a_\infty(g)(x)=\limsup_{y\to\infty}\frac{g(x,y)}{y}.
\end{equation}
We set  $a_0^k(g):=a_0(g)\wedge k$, $a_\infty^k(g):=a_\infty(g)\vee(-k)$. 
For brevity we let $a_0:=a_0(f)$, $a_\infty:=a_\infty(f)$, $a_0^k:=a^k_0(f)$, $a^k_\infty:=a^k_\infty(f)$.

{\bf Generalized principal eigenvalues.} 
If $[a_0(g)]^-,[a_\infty(g)]^+\in L^\infty(E;m)$, we set 
\[
\lambda_1^L(a_0^k(g)):=\lambda_1(-L-a_0^k(g)),\quad \lambda_1^L(a_\infty^k(g)):= \lambda_1(-L-a_\infty^k(g)).
\]
By Corollary \ref{cor.pert1} $\lambda_1^L(a_0^k(g))\ge \lambda_1^L(a_0^{k+1}(g))$ and   
$\lambda_1^L(a_\infty^k(g))\le \lambda_1^L(a_\infty^{k+1}(g))$ for $k\ge 1$. We let 
\[
\lambda_1^L(a_0(g)):=\inf_{k\ge 1}\lambda_1^L(a_0^k(g)),
\qquad
\lambda_1^L(a_\infty(g)):=\sup_{k\ge 1}\lambda_1^L(a_\infty^k(g)).
\]
In what follows, we mostly omit the superscript $L$ from the notation of the corresponding eigenvalues, 
except where it is useful for clarity.

\begin{remark}
\label{rem.shift1}
If the spectral bound of the operator  under consideration is non-negative, we choose
$\kappa\ge 0$ sufficiently large so that
\[
s(L_\kappa)<0,
\qquad
L_\kappa:=L-\kappa I,
\]
and apply the existence results below to the shifted equation
\begin{equation}\label{eq1.1l}
    -L_\kappa u=f^\kappa(\cdot,u)\quad\text{in }E,
\end{equation}
where
\[
f^\kappa(x,y):=f(x,y)+\kappa y.
\]
We can do this because $(L_\kappa,\mathscr D(L))$ generates the positive
$C_0$-semigroup $(e^{-\kappa t}T_t)_{t\ge0}$, and its resolvent is
compact and positivity improving.  Moreover,
$f^\kappa$ satisfies \textup{(F1)--(F3)}, and 
\[
a_0(f^\kappa)=a_0+\kappa,
\qquad
a_\infty(f^\kappa)=a_\infty+\kappa,
\]
which, in turn,  yields that 
the operators determining the corresponding spectral conditions remain
unchanged. 
Indeed, for every $k>\kappa$,
\[
-L_\kappa-a_0(f^\kappa)\wedge k
=
-L-a_0\wedge(k-\kappa).
\]
Similarly, for every $k\ge1$,
\[
-L_\kappa-a_\infty(f^\kappa)\vee(-k)
=
-L-a_\infty\vee(-(k+\kappa)).
\]
Letting $k\to\infty$ in the definitions of the generalized principal
eigenvalues yields
\[
\lambda_1^{L_\kappa}(a_0(f^\kappa))
=
\lambda_1^{L}(a_0),
\qquad
\lambda_1^{L_\kappa}(a_\infty(f^\kappa))
=
\lambda_1^{L}(a_\infty).
\]
Therefore, although the generalized slopes are shifted by $\kappa$, the
spectral conditions used in the existence results are invariant under this
transformation.
\end{remark}

An important   technical point is that, under the above assumptions and by
 the Rohlin theorem, $R_\alpha$ admits a kernel representation on $(E,\mathcal A)$
for every  $\alpha>s(L)$. More precisely,  there exists a family $\{r_\alpha(x,dy),\, x\in E\}$
such that $x\mapsto r_\alpha(x,A)$ is $\mathcal A$-measurable for every 
$A\in\mathcal A$, $r_\alpha(x,dy)$ is a  measure on $\mathcal A$ 
for every  $x\in E$, and,  for any positive  function $g\in L^p(E;m)$
\begin{equation}
\label{eq.rochlin}
\widetilde{R_\alpha g}(x)=\int_E \tilde g(z)\,r_\alpha(x,dz)\quad m\text{-a.e.},
\end{equation}
where $\widetilde{R_\alpha g}$ and $\tilde g$ are representatives of the 
$m$-equivalence classes of $R_\alpha g$ and $g$, respectively.
Consequently,  $R_\alpha$ extends naturally to positive $\mathcal A$-measurable functions on  $E$  (\cite[Theorem 2.1]{Simmons} 
or \cite[Theorem 2.2]{BBR}).  This feature is central in our framework because 
in general the negative part of the nonlinearity in \eqref{eq1.1} 
is not automatically controlled in $L^p$.

To simplify notation, we omit the tildes in identities of the form
\eqref{eq.rochlin} and adopt the following convention.

\begin{remark}
\label{rem.not}
Let $\mathcal X$ be a set  of pointwise-defined $\mathcal A$-measurable
functions on $E$, and let $\mathcal Z$ be a set  of $\mu$-equivalence
classes of functions from $\mathcal X$, for a given measure $\mu$ on $\mathcal A$.
Suppose that  $A:\mathcal X\to\mathcal X$,  $B:\mathcal Z\to \mathcal Z$  are linear operators.
Throughout the paper, whenever we compare such operators, we shall use the
following convention. If $f\in \mathcal Z$, then an identity written in the
form
\[
Af = Bf
\]
means that, after choosing a representative $\tilde f\in\mathcal X$ of the
$\mu$-equivalence class $f$, the function $A\tilde f$ is a representative
of the $\mu$-equivalence class $Bf$.
Similarly, if $g\in\mathcal X$ and the $\mu$-equivalence class  $[g]_{\mu}$ belongs to $\mathcal Z$, then an identity
written in the form
\[
Ag = Bg
\]
is understood as
\[
[Ag]_{\mu} = B[g]_{\mu}.
\]
Observe that the assertion $Af=Bf$ for $f\in\mathcal Z$ entails the following
property of the operator $A$: whenever $f,g\in\mathcal X$ satisfy
$f=g$ $\mu$-a.e., one has $Af=Ag$ $\mu$-a.e.
\end{remark}
  
\begin{remark}[Completion of the underlying measure space]
The results concerning $L^p$-solutions remain valid when the
underlying measure space is replaced by its completion.
Indeed, let  $(E,\mathcal A^m,\bar m)$ denote the completion of $(E,\mathcal A,m)$.
The canonical identification
\[
L^p(E,\mathcal A,m)
=
L^p(E,\mathcal A^m,\bar m)
\]
is an isometric lattice isomorphism. Thus the semigroup,
its generator and resolvent, and the corresponding spectral
quantities are unchanged under this identification.

Moreover, if $f$ is measurable in $x$ with respect to
$\mathcal A^m$ and continuous in $y$, it admits a
$\mathcal A$-measurable Carath\'eodory version $\widetilde f$
such that, outside a single $m$-null set,
\[
\widetilde f(x,y)=f(x,y)
\qquad\text{for every }y\ge0.
\]
This follows by choosing measurable versions at nonnegative
rational values of $y$ and using continuity.
Consequently, the assumptions on $f$, the asymptotic slopes,
and the notion of solution are preserved.

In particular, using the Lebesgue $\sigma$-algebra instead of
the Borel $\sigma$-algebra on a Euclidean domain does not
affect the $L^p$-problem considered here. The probabilistic
realization is carried out on the underlying Lusin measurable space.
\end{remark}

\section{Right Markov processes and  potential theory}
\label{sec2}

Recall that a topological space is Lusin if it is homeomorphic to a Borel subset of a Polish space. 
Let $Y$ be a Lusin topological space and set
\[
Y_\partial:=Y\cup\{\partial\},
\]
where $\partial$ is an isolated point. We denote by $\mathcal{B}(Y_\partial)$ the Borel $\sigma$-algebra of 
$ Y_{\partial}$ and, for $B\in \mathcal{B}(Y_\partial)$, set
\[
\mathcal B(B):=\{C\in\mathcal B(Y_\partial): C\subset B\}.
\]
Every numerical function on $Y$, with values in
$\mathbb R\cup\{+\infty\}\cup\{-\infty\}$, is tacitly extended to $Y_\partial$  by assigning the value $0$ at $\partial$. 
The same notation $\mathcal B(B)$ will also be used for the class of numerical $\mathcal B(B)$-measurable functions on $B$.
A subset $A\subset Y_\partial$ is said to be universally measurable if, for every positive finite measure $\mu$ 
on $\mathcal B(Y_\partial)$, there exist $A_1,A_2\in\mathcal B(Y_\partial)$ such that
\[
A_1\subset A\subset A_2,
\qquad
\mu(A_2\setminus A_1)=0.
\]
The universally measurable subsets of  $B$ form a $\sigma$-algebra, denoted by $\mathcal{B}^u(B)$. 
Every $\sigma$-finite positive measure on $\mathcal B(B)$ admits a unique extension to $\mathcal B^u(B)$. 
Finally, if $\mathcal H$ is a class of numerical functions, we denote by $p\mathcal H$ and $b\mathcal H$ 
its subclasses of nonnegative and  bounded functions, respectively.
It is readily verified that a positive numerical function   $g$ on $B$  belongs to $p \mathcal{B}^u(B)$
 if and only if for every finite measure $\mu$ on $\mathcal{B}(B)$ there exist  
 $g_1, g_2 \in p \mathcal{B}(B)$ such that $g_1 \leq g \leq g_2$ and $\mu(g_1<g_2)=0$.

Let $(\Omega, \mathcal{F})$ be a measurable space. 
A function $X:[0,\infty]\times\Omega\to Y_\partial$ is called a  stochastic process on 
$(\Omega, \mathcal{F})$ with state space $Y$ if for any  $t\in [0,\infty]$ the mapping 
$X(t,\cdot)$ is  $\mathcal{F} / \mathcal{B}(Y_\partial)$-measurable,  
$X(\infty,\omega)=\partial$ for all $\omega \in \Omega$, 
and the following property holds for any $\omega\in\Omega$:  
if  for some $t_0\ge 0$, $X(t_0,\omega)=\partial$,  then $X(t,\omega)=\partial$ 
for any  $t\ge t_0$. In what follows we use the notation $X_{t}(\omega):= X(t,\omega)$
for any stochastic process $X$. The function $\zeta: \Omega \longrightarrow[0, \infty]$ defined by 
\[
\zeta(\omega)=\inf \{t\ge 0: X_t(\omega)=\partial\}
\]
is called the lifetime of the stochastic process $X$.

A stochastic process $X$ on $(\Omega, \mathcal{F})$ with state space $Y$
 is called right-continuous if for every $\omega \in \Omega$ the function 
 $t \longmapsto X_t(\omega)$ is right-continuous on $[0, \infty)$.
 
 A filtration on a measurable space $(\Omega, \mathcal{F})$  is a family 
 $\left(\mathcal{F}_t\right)_{t \geq 0}$ of $\sigma$-algebras $\mathcal{F}_t$
 included in  $\mathcal{F}$ such that $\mathcal{F}_s \subset \mathcal{F}_t$ whenever  $s\le t$. 
 We let  $\mathcal{F}_{\infty}:=\sigma(\bigcup_{s<\infty} \mathcal{F}_s)$. 
For every $t \in \mathbb{R}_{+}$ we put $\mathcal{F}_{t+}=\bigcap_{s>t} \mathcal{F}_s$. 
A  filtration $\left(\mathcal{F}_t\right)_{t \geq 0}$ 
is called right-continuous provided that $\mathcal{F}_t=\mathcal{F}_{t+}$ for all $t \geq 0$.
A  stochastic process $X$ on $(\Omega, \mathcal{F})$
with state space $Y$ is called $\left(\mathcal{F}_t\right)$-adapted 
provided that the map $X_t$ is $\mathcal{F}_t / \mathcal{B}(Y_\partial)$-measurable for any  $t \geq 0$. 

A transition function on $Y_\partial$ is a family of kernels $(P_t)_{t \geq 0}$
on $(Y_\partial, \mathcal{B}^u(Y_\partial))$ 
such that $P_t1= 1$, and  $P_t(P_s f)=P_{s+t} f$
for all $f \in p \mathcal{B}^u(Y_\partial)$ and $s, t \in \mathbb{R}_{+}$. 

Let $(\Omega, \mathcal{F})$ be a measurable space and   $\left(P_t\right)_{t \geq 0}$
be a transition function on $Y_\partial$.  
A collection $\mathbb X=\left( \Omega, \mathcal{F}, (\mathcal{F}_t)_{t\ge 0},X, (\theta_t)_{t\ge 0}, 
(\mathbb P_x)_{x\in Y_\partial}\right)$
 is called a Markov process with state space $Y$ and  transition function $\left(P_t\right)_{t \geq 0}$ if:
 \begin{itemize}
\item[(i)] the family $\left(\mathcal{F}_t\right)_{t \geq 0}$ is a filtration on 
$(\Omega, \mathcal{F})$ and  $X$ is a $\left(\mathcal{F}_t\right)$-adapted
stochastic  process with state space $Y$;
\item[(ii)]$\theta_t:  \Omega \longrightarrow  \Omega$ 
and $X_s \circ \theta_t=X_{s+t}$ for all $s,t\ge 0$;
\item[(iii)] $\mathbb P_x$ is a probability measure on $(\Omega, \mathcal{F})$ for any $x \in Y_\partial$, and 
the mapping   $x \longmapsto \mathbb P_x(A)$ is $\mathcal{B}^u(Y_\partial)$-measurable for any  $A \in  \mathcal{F}$;
\item[(iv)]
for any $f\in p \mathcal{B}^u(Y)$, $s, t \geq 0$, $G \in \mathcal{F}_s$,
\[
\mathbb E_x\left(f(X_{s+t}) \mathbf 1_G\right)= 
\mathbb E_x\left(P_t f(X_s) \mathbf 1_G\right),\quad \mathbb E_x\left(f(X_0)\right)=f(x),\quad x \in Y_\partial.
\] 
\end{itemize}

For any probability measure   $\mu$ on $\mathcal{B}(Y_\partial)$ we let 
$\mathbb P_\mu$ denote  the probability measure on $(\Omega, \mathcal{F})$ 
given by $\mathbb P_\mu(F):=\int \mathbb P_x(F) d \mu(x),\, F \in \mathcal{F}$. 
We denote by $\mathcal{F}^\mu$ the completion of $\mathcal{F}$ with respect to $\mathbb P_\mu$. 
Next,  for each $t \geq 0$ we let $\mathcal{F}_t^\mu$ denote the completion of $\mathcal{F}_t$
 in $\mathcal{F}^\mu$ with respect to $\mathbb P_\mu$. We set
$$
\widetilde{\mathcal{F}}_t:=\bigcap_\mu \mathcal{F}_t^\mu \quad, \quad \widetilde{\mathcal{F}}:=\bigcap_\mu \mathcal{F}^\mu,
$$
where both  intersections run over the set of all probability measures $\mu$ on $\mathcal{B}(Y_\partial)$.

If $\mathbb X$ is a right-continuous Markov process with transition function $(P_t)$, then 
one easily sees that for every 
bounded continuous function $f$ on $Y$   and $x \in Y$ the function $t \longmapsto P_t f(x)$
is right-continuous and therefore the function $(t, x) \longmapsto P_t f(x)$ is 
$\mathcal{B}([0, \infty)) \otimes \mathcal{B}^u(Y_\partial)$-measurable. 
Thus, we may define the kernel  $\mathcal{U}=\left(U_\alpha\right)_{\alpha>0}$ on $(Y_\partial, \mathcal{B}^u(Y_\partial))$
 by
$$
U_\alpha f(x)=\int_0^{\infty} e^{-\alpha t} P_t f(x) d t, \quad \text { for all } f \in p \mathcal{B}^u(Y_\partial),\, x\in Y_\partial.
$$
The family $\mathcal{U}$ is called the resolvent kernel associated with the Markov process $\mathbb X$.

\begin{definition}
A function $u\in p \mathcal{B}^u(Y)$ is called $\alpha$-excessive ($\alpha\ge 0$) with respect to $\mathcal U$ if
\begin{enumerate}
\item[(i)] $\beta U_{\alpha+\beta}u(x)\le u(x),\, x\in Y,\, \beta>0$;
\item[(ii)] $\beta U_{\alpha+\beta}u(x)\nearrow u(x),\, x\in Y$ as $\beta\to \infty$.
\end{enumerate}
\end{definition}
For $0$-excessive functions with respect to $\mathcal U$, we simply use the
term $\mathcal U$-excessive.

\begin{definition}
A right-continuous Markov process $\mathbb X$ with state space $Y$ and 
transition function $\left(P_t\right)_{t \geq 0}$ is called a right (Markov) process if the following properties hold:
\begin{enumerate}
\item[(i)] the filtration $\left(\mathcal{F}_t\right)_{t \geq 0}$ is right-continuous and  $\mathcal{F}_t=\widetilde{\mathcal{F}}_t, t\ge 0$.
\item[(ii)] for each $\alpha$-excessive function $u$ (with respect to $\mathcal{U}$) 
and every probability measure   $\mu$ on $\mathcal B(Y)$ the function $t \longmapsto u(X_t)$ is right-continuous on 
$[0, \infty]$, $\mathbb P_\mu$-a.e.
\end{enumerate}
\end{definition}

\begin{definition}
\label{def.rep}
Let $Y$ be a Lusin topological space and $\mu$ be a finite Borel measure on $Y$. Assume that $(S_t)$
is a positive $C_0$-semigroup on $L^p(Y;\mu)$ with generator $(A,\mathscr D(A))$ 
and the resolvent $(G_\alpha)$, and that  $s(A)<0$. 
We say that a right Markov process $\mathbb X$ with state space $Y$ and a transition function $(P_t)$
{\em represents } $(S_t)$ (or represents $(G_\alpha)$) if for any  $f\in p\mathcal B(Y)$, with $\int_Y|f|^p\,d\mu<\infty$,
and any $\alpha>s(A)$,
\begin{equation}
\label{eq.equiv123}
U_\alpha f=G_\alpha f,\quad \mu\text{-a.e.}
\end{equation}
(see the convention adopted in Remark \ref{rem.not}).
\end{definition}
Let us now recall the notion of terminal  times and additive functionals.  
Let $\mathcal U=(U_\alpha)$ be a resolvent kernel on $Y$.
A $\sigma$-finite Borel measure $\mu$ on $Y$
is called $\mathcal U$-excessive if for any positive function $g\in \mathcal B(Y)$, we have 
\[
\int_Y\alpha U_\alpha g\,d\mu\le \int_Y g\,d\mu.
\]
A set $A\subset Y$ is called $\mu$-polar if there exists a $\mathcal U$-excessive function $u$ on 
$Y$ such that $A\subset \{u=\infty\}$ and $\mu(\{u=\infty\})=0$.

\begin{definition} 
Let $\mathbb X$ be a right Markov process on a Lusin topological space $Y$.
\begin{enumerate}
\item[(i)] A random variable $\tau:\Omega\to [0,\infty]$ is called a stopping time if $\{\tau\le t\}\in \mathcal F_t$ for every $t\ge 0$.
\item[(ii)]
An   $(\mathcal F_t)$-adapted  process $(A_t)$,  with values in $[0,\infty]$ 
and $A_0=0$, is called a {\em positive  additive functional of $\mathbb X$} 
if  there exist  $\Omega_d \in \mathcal{F}$ (called a defining set for $(A_t)$), and a Borel 
$\mu$-polar set $N$ (called an exceptional set for $(A_t)$) such that
\begin{itemize}
\item $\mathbb{P}_x(\Omega_d)=1$,  $x \notin N$;
\item $\theta_t \Omega_d \subset \Omega_d$ for all $t \geq 0$;
\item for all $\omega \in \Omega_d$ the mapping $t \mapsto A_t(\omega)$ 
is non-decreasing and right-continuous on $[0, \infty)$ and finite-valued on $[0, \zeta(\omega))$;
\item for all $\omega \in \Omega_d$ and $s, t \geq 0$: $A_{t+s}(\omega)=A_t(\omega)+A_s\left(\theta_t \omega\right)$.
\end{itemize}
\item[(iii)] A stopping time $\tau$ is called an exact  terminal time provided that
for any  Borel probability  measure $\mu$ on $Y$ and each stopping time $\sigma$
we have 
\[
\tau=\sigma+\tau \circ \theta_{\sigma},\quad \mathbb P_\mu\text{-a.s. on } \{\tau>\sigma\},
\]
and whenever $t_n\searrow 0$, then $t_n+\tau\circ \theta_{t_n}\to \tau$ $\mathbb P_\mu$-a.s.
\end{enumerate}
\end{definition}

\begin{remark}
\label{rem.main}
Let $(E,\mathcal A,m)$ be a Lusin measure space as introduced at the beginning of Section \ref{sec.prelim}.
Suppose that  the resolvent $(R_\alpha)$ of $(L,\mathscr D(L))$ is sub-Markovian, i.e. $\alpha R_\alpha1\le 1$
for any $\alpha>0$. 
Then, by  \cite[Theorem 2.2]{BBR} there exists a   Lusin topological space 
$\hat E$ and  a right Markov process $\hat{\mathbb  X}$ with state space 
$\hat E$ and the  resolvent kernel  $\hat{\mathcal U}:= (\hat U_\alpha)_{\alpha>0}$, such that:
\begin{enumerate}[(i)]
\item  $E \subset \hat E$, 
$E\in \mathcal{B}(\hat E)$, $\mathcal A=\{B\cap E:B\in\mathcal B(\hat E)\}$, 
\item Let $\bar{m}$ be  the extension of $m$ to  $\mathcal{B}(\hat E)$ defined by $\bar m(\hat E\setminus E)=0$. 
Then each $\hat U_\alpha$ is well defined as an  operator  on  
$L^p(\hat E, \bar{m})$, i.e. 
whenever $f,g\in p\mathcal B^u(\hat E)$,  and $f=g$ $\bar{m}$-a.e., 
we have $\hat U_\alpha f=\hat U_\alpha g$ $\bar m$-a.e. 
\item Let $\hat R_\alpha$ be the trivial extension of $R_\alpha$ to $L^p(\hat E,\bar m)$, i.e.
for $f\in L^p(\hat E,\bar m)$, we let $\hat R_\alpha f:= R_\alpha (f\mathbf1_E)$ on $E$ and zero on 
$\hat E\setminus E$. Then the  operators  $\hat U_\alpha$ and $\hat R_\alpha$  
coincide on $L^p(\hat E,\bar m)$   for any $\alpha> 0$
(see Remark \ref{rem.not}).  
\end{enumerate}
By  Definition \ref{def.main}, solving \eqref{eq1.1} amounts to finding  $u\in L^p(E;m)$ such that 
\begin{equation}
\label{eq.transf1}
u+Rf^-(\cdot,u)=Rf^+(\cdot,u)\quad \text{in }L^p(E;m).
\end{equation}
If $u$ is a  solution to \eqref{eq.transf1}, then its extension 
$\hat u\in L^p(\hat E;\bar m)$, defined  by    $\hat u(x):= u(x),\, x\in E$,  $\hat u(x):=0,\, x\in \hat E\setminus E$, satisfies
\begin{equation}
\label{eq.transf2}
\hat u+\hat R\hat{f}^-(\cdot,\hat u)=\hat R\hat{f}^+(\cdot,\hat u)\quad \text{in }L^p(\hat E;\bar m).
\end{equation}
Here $\hat f$ is defined by $\hat f(x,y):= f(x,y),\, x\in E,\,y\in\mathbb R$, $\hat f(x,y):=0,\, x\in \hat E\setminus E,\,y\in\mathbb R$.
Conversely,  if  $\hat u\in L^p(\hat E;\bar m)$ solves \eqref{eq.transf2}, then  $u:= \hat u_{|E}$
satisfies \eqref{eq.transf1}. Clearly, $(\hat R_\alpha)$ is a compact  positivity-improving 
resolvent on $L^p(\hat E;\bar m)$,  $\hat f$ satisfies (F1)-(F3) on $L^p(\hat E;\bar m)$, 
and $\lambda_1(a_0(\hat f))=\lambda_1(a_0)$, $\lambda_1(a_\infty(\hat f))=\lambda_1(a_\infty)$.

It is therefore justified to state, when analyzing the problem \eqref{eq1.1},  that, without loss of generality, 
one may assume that  $E$ is a Lusin topological space, $\mathcal A=\mathcal B(E)$, and  that 
there exists a right Markov process 
$\mathbb X$  that represents  the semigroup $(T_t)$.
\end{remark}

\section{Deny's theorem for sub-Markovian semigroups}
\label{sec3}
In what follows, for a given Banach space $(F,\|\cdot\|_F)$, we denote by $\langle\cdot,\cdot\rangle_F$ 
the duality pairing between $F$
and $F^*$. Note that throughout the paper the space $L^p(E;m)$ is fixed. Therefore, we simply write 
$\langle\cdot,\cdot\rangle$ for $\langle\cdot,\cdot\rangle_{L^p(E;m)}$, and $\|\cdot\|$ for $\|\cdot\|_{L^p(E;m)}$.

\begin{definition}
We say that a positive function $u\in L^p(E;m)$ is supermedian if $T_tu\le u$
for any $t\ge 0$.
\end{definition}
A simple example of a supermedian function is $Rf$ for any positive $f\in L^p(E;m)$.

\begin{lemma}
\label{lm1}
Assume  that  $(T_t)$ is sub-Markovian. Then, for any  concave function $\phi:[0,\infty)\to [0,\infty)$ 
and any  supermedian function $u$, the function $\phi(u)$ is supermedian.  
\end{lemma}
\begin{proof}
Since $\phi$ is concave,  there exist $a,b\in\mathbb R$ such that $\phi(y)\le ay+b,\, y\ge 0$.
Therefore, $\phi(u)\in L^p(E;m)$. Furthermore, we have $\phi=\inf_{q\in\mathbb Q} l_q$, 
where $(l_q)$ is a family of affine  functions. Since $\phi$ is nonnegative and concave on $[0,\infty)$
the affine majorants $l_q(y)=a_qy+b_q$ may be chosen with $a_q,b_q\ge0$.
Thus,
\[
T_t\phi(u)= T_t[\inf_{q\in\mathbb Q} l_q(u)]\le \inf_{q\in \mathbb Q}T_t[l_q(u)]\le \inf_{q\in \mathbb Q}l_q(u)=\phi(u).
\]
In the last inequality we used the facts that $u$ is supermedian and $(T_t)$ is sub-Markovian. This proves the assertion.
\end{proof}

\begin{theorem}
\label{lm.upseq}
Assume that $(T_t)$ is sub-Markovian. For any sequence $(v_n)$ of
supermedian functions there exists a subsequence $(v_{n_k})$ which
converges $m$-a.e. in $E$, possibly to $+\infty$.
\end{theorem}
\begin{proof}
Set $w_n:=v_n/(1+v_n)$.
By Lemma \ref{lm1}, $(w_n)$ is a sequence of supermedian functions
satisfying $0\le w_n\le1$.  We first apply the argument below to the
sequence $(w_n)$.
After passing to a subsequence,   $(w_n)$ converges   weakly$^\star$  
in $\sigma(L^\infty,L^1)$. Let us denote the limit by $w$. By the assumptions made, 
the resolvent of $L$ is compact; hence, after passing to a further  
subsequence, $(Rw_n)$  converges in $L^p(E;m)$ to a function $\varrho\in L^p(E;m)$. 
We have for any $\eta\in L^\infty(E;m)$,
\[
\langle \varrho,\eta\rangle=\lim_{n\to \infty}\langle Rw_n,\eta\rangle=\lim_{n\to \infty}\langle w_n,R^*\eta\rangle
=\langle w,R^*\eta\rangle=
\langle Rw,\eta\rangle.
\] 
Thus, $(Rw_n)$ converges to $Rw$ in $L^p(E;m)$,  
and by the resolvent identity $R_\alpha w_n\to R_\alpha w$ in $L^p(E;m)$ for any $\alpha\ge 0$.
By the diagonal method one can choose a further subsequence such that $R_\alpha w_n\to R_\alpha w$ a.e.
for any $\alpha\ge 0,\alpha\in\mathbb Q$. Consequently,
\[
\liminf_{n\to \infty} w_n\ge \liminf_{n\to \infty} \alpha R_\alpha w_n
= \alpha R_\alpha w\quad \text{a.e. for }\alpha\ge 0,\alpha\in\mathbb Q.
\]
Letting $\alpha\to \infty$, we   conclude that
\[
\liminf_{n\to \infty} w_n\ge w,\quad \text{a.e.}
\]
By \cite[Theorem 26, page 28]{DellacherieMeyer}, after passing to a subsequence,  $w_n\to w$ a.e. 
Finally, since the map $t\mapsto t/(1+t)$ is a homeomorphism from
$[0,\infty]$ onto $[0,1]$, we conclude that
\[
v_{n_k}\longrightarrow \frac{w}{1-w}
\quad m\text{-a.e.},
\]
with the convention that $w/(1-w)=+\infty$ on $\{w=1\}$.
\end{proof}

\section{Kato's inequality for sub-Markovian semigroups}
\label{sec4}

We recall a few standard notions from probability theory. 
A stochastic process $(X_t)$  on a probability space  $(\Omega,\mathcal H, \mathbb P)$, 
with state space $Y$, is called c\`adl\`ag if, for $\mathbb P$-a.e. $\omega\in \Omega$, the path
$t\mapsto X_t(\omega)$ is right-continuous on $[0,\infty)$ and admits left limits in $Y$
at every $t>0$. We write
\[
X_{t-}(\omega):=\lim_{s\uparrow t}X_s(\omega),
\qquad t>0.
\]
If $Z_t=u(X_t)$  is c\`adl\`ag, we use the notation
\[
u(X)_{t-}:=Z_{t-}.
\]
In general, $u(X)_{t-}$ need not coincide with $u(X_{t-})$. The right Markov processes considered 
below need not themselves have c\`adl\`ag paths; nevertheless, Lemma \ref{lm.basic}(2) shows that $u(X)$
is c\`adl\`ag whenever $u$ is a finite $\mathcal U$-excessive function.

\begin{definition} 
Fix a probability space $(\Omega,\mathcal H, \mathbb P)$ with a filtration $(\mathcal H_t)$.
\begin{enumerate}
\item[(i)] A c\`adl\`ag $(\mathcal H_t)$-adapted stochastic process $(M_t)$
is called  a martingale if $\mathbb E |M_t|<\infty$ for each  $t\ge 0$, and $\int_A M_t\,d\mathbb P=\int_A M_s\,d\mathbb P$
for any $A\in\mathcal H_s$ and $s\le t$.
\item[(ii)] A random variable $\tau:\Omega\to [0,\infty]$ is called a stopping time if $\{\tau\le t\}\in \mathcal H_t$ for every $t\ge 0$.
\item[(iii)] A c\`adl\`ag $(\mathcal H_t)$-adapted stochastic process $(M_t)$
is called  a {\em local martingale} if there exists an increasing sequence $(\tau_k)$ of stopping times such that
$\tau_k\nearrow \infty$ and for each $k\ge 1$, the  process $M^k:=M_{\cdot\wedge\tau_k}$ is a martingale. 
The sequence $(\tau_k)$ is called a {\em localizing sequence}.
\item[(iv)] A martingale  $(M_t)$ with respect to  $(\mathcal H_t)$ is called {\em uniformly integrable}
if the family of random variables $(M_t)_{t\ge 0}$ is uniformly integrable with respect to $\mathbb P$.
\end{enumerate}
\end{definition}

Note that all of the above notions depend on a given filtration and probability measure. 
If the underlying probability measure or filtration is not clear from the context, we will state it explicitly.

\begin{lemma}
\label{lm.basic}
Let $Y$ be a  Lusin topological space. Suppose that $(S_t)$ is a sub-Markovian $C_0$-semigroup on  
$L^p(Y,\mathfrak m)$ with  generator $(A,\mathscr D(A))$, where  $\mathfrak m$ is a bounded measure on $\mathcal B(Y)$.  
Furthermore, suppose that the resolvent $(G_\alpha)$  of $(A,\mathscr D(A))$  is compact and positivity improving.
Let $\mathbb X$ be a right Markov process related to  $(G_\alpha)$ (see Remark \ref{rem.main}) 
with resolvent kernel $\mathcal U$.  
\begin{enumerate}
\item[(1)]  If $u$ is a supermedian function with respect to $(S_t)$, then there exists an $\mathfrak m$-version $\tilde u$
of $u$ such that $\tilde u$ is $\mathcal U$-excessive.
\item[(2)] If $u$ is a $\mathcal U$-excessive function, then for any $x\in \{u<\infty\}$ there exist
a local martingale $M^x$ and an increasing c\`adl\`ag process $A^x$, with $M^x_0=A^x_0=0$, such that
$\mathbb E_x A^x_\infty<\infty$, and 
\[
u(X_t)=u(X_0)-A^x_t+M^x_t,\quad t\ge 0,\, \mathbb P_x\text{-a.s.}
\]
Consequently, for any $x\in \{u<\infty\}$ and localizing sequence $(\tau_k)$ for a local martingale $M^x$, we have
\[
u(x)=\lim_{k\to \infty} (\mathbb E_xu(X_{\tau_k})+\mathbb E_xA^x_{\tau_k}).
\]
\item[(3)] Let   $g\in p\mathcal B(Y)$ and   $u:= Ug$ (here $U:=U_0$).
Then $u$ is  $\mathcal U$-excessive,  and  for any $x\in\{u<\infty\}$
there exists a uniformly integrable martingale $M^x$, with $M^x_0=0$, such that  
\begin{equation}
\label{eq.marun1}
u(X_t)=u(x)-\int_0^tg(X_s)\,ds+M^x_t,\quad t\ge 0\,\mathbb P_x\text{-a.s.}
\end{equation}
Furthermore, for any increasing sequence $(\tau_k)$ of stopping times such that $\tau_k\nearrow \infty$,
\[
\mathbb E_xu(X_{\tau_k})\searrow 0,\quad k\to \infty.
\]
\item[(4)] If $u$ in (2) is finite $\mathfrak m$-a.e. then there exists a positive  additive functional 
$(A_t)$, with exceptional set $N:=\{u=\infty\}$, such that
$A_t=A^x_t,\, t\ge 0\, \mathbb P_x$-a.s. for any  $x\in \{u<\infty\}$.
\item[(5)]  If $u$ is $\mathcal U$-excessive, then for any increasing sequence $(\tau_k)$ of exact terminal  times the function
\[
v(x):=\sup_{k\ge 0} \mathbb E_x u(X_{\tau_k}),\quad x\in Y
\]
is a $\mathcal U$-supermedian function.
\item[(6)] If $u$ is an  $\mathcal U$-excessive function, then the set $B:=\{u<\infty\}$ is absorbing, i.e.
for any $x\in B$ we have $\mathbb P_x(X_t\in B,\,\text{ for every } t\ge 0)=1$.
\end{enumerate}
\end{lemma}
\begin{proof}
Let $(\lambda_1(-A),\eta_1^*)$ be a dual principal eigenpair related to $(A,\mathscr D(A))$ (see Proposition \ref{prop.tool}(v)). 
Let $\hat{\mathfrak m}:= \eta_1^*\mathfrak m$. Observe that
\[
\int_Y \alpha U_\alpha f \,d\hat{\mathfrak m}\le  \int_Y f\, d \hat{\mathfrak m}
\]
for any  $f\in p\mathcal B(Y)$. Thus 
$\hat{\mathfrak m}$ is a $\mathcal U$-excessive measure that is equivalent to $\mathfrak m$.
Consequently, we may apply \cite[Proposition 2.4]{BCR}, which proves  (1). 

By the very definition of a right Markov process and  \cite[Theorem III.5.7]{BG} we have that  for 
$x\in \{u<\infty\}$, $(u(X_t))_{t\ge 0}$
is a c\`adl\`ag $\mathbb P_x$ supermartingale,  $\sup_{t\ge 0} \mathbb E_xu(X_t)\le u(x)<\infty$ and 
$\mathbb P_x(u(X_t)<\infty,\, t\ge 0)=1$. Thus (2) is a consequence of the Doob-Meyer decomposition theorem. 
As to (3), first note that $u$ is $\mathcal U$-excessive by \cite[Proposition II.2.2(e)]{BG}. Thus, by (2) and the Markov property 
for any $x\in \{u<\infty\}$,
\begin{equation}
\label{eq.marun12}
u(X_t)=\mathbb E_x\left(\int_t^\infty g(X_s)\,ds\Big|\mathcal F_t\right)
=-\int_0^tg(X_s)\,ds+u(x)+M^x_t,\quad t\ge 0,\,\mathbb P_x\text{-a.s.,}
\end{equation}
where $M^x_t:=\mathbb E_x\left(\int_0^\infty g(X_s)\,ds\Big|\mathcal F_t\right)-u(x)$. 
Clearly $M^x$ is a uniformly integrable martingale. In particular $\mathbb E_xM^x_{\tau_k}=0,\, k\ge 1$. 
Consequently, by \eqref{eq.marun12},
\[
u(x)=\mathbb E_xu(X_{\tau_k})+\mathbb E_x\int_0^{\tau_k}g(X_s)\,ds.
\]
Letting $k\to \infty$, we deduce that $\mathbb E_xu(X_{\tau_k})\searrow 0$, 
since $u(x)=\mathbb E_x\int_0^{\infty}g(X_s)\,ds$.  This proves (3). 
Assertion (4) follows from \cite[Theorem 51.7]{Sharpe}, after a standard   argument restricting 
$\mathbb X$ to $\{u<\infty\}$; see \cite[Theorem 12.23]{Sharpe}. 
Finally,  (5) follows from \cite[Exercise (2.18)]{BG} and directly from the definition of a supermedian function.
The property (6) can be found in \cite[Lemma 5.2]{FG}.
\end{proof}

\begin{proposition}
\label{lm2}
Assume that $(T_t)$ is sub-Markovian.   Suppose that  $v,w$ are  supermedian functions.
Set $u:=Rf-w$  for some $f\in L^p(E;m)$.
Then
\[
(u-v)^+\le R(\mathbf1_{\{u>v\}}f).
\]
\end{proposition}
\begin{proof}
By Remark \ref{rem.main} we may assume that $E$ is a Lusin topological space, $\mathcal A=\mathcal B(E)$, 
and  there exists a right Markov process $\mathbb X$  that represents 
the semigroup $(T_t)$ (see Definition \ref{def.rep}). We  fix a version of $f$.   
Thus, by Lemma \ref{lm.basic}(1)  there exist $m$-versions $\tilde v, \tilde w$ of $v, w$ such that 
$\tilde v, \tilde w$ are $\mathcal U$-excessive, and furthermore,  for $m$-a.e. $x\in E$ there exist  increasing 
c\`adl\`ag processes  $(A^x_t)$,  $(B^x_t)$, and  local martingales  
$(N^x_t)$, $(M^x_t)$, with $A^x_0=B^x_0=M^x_0=N^x_0=0$, such that
\[
\tilde u(X_t)=\tilde u(X_0)-\int_0^t f(X_s)\,ds+B^x_t+M^x_t,\quad t\ge 0,\,\mathbb P_x\text{-a.s.},
\]
for an $m$-version $\tilde u$ of $u$ (to be precise $\tilde u:= Uf-\tilde w$),
and 
\[
\tilde v(X_t)=\tilde v(X_0)+N^x_t-A^x_t,\quad t\ge 0,\,\mathbb P_x\text{-a.s.}
\]
Now applying the It\^o--Meyer formula to $\tilde u(X_t)-\tilde v(X_t)$ and the convex function $x\mapsto x^+$
(see \cite[Theorem IV.70]{Protter}) we obtain that for $m$-a.e. $x\in E$,
\[
\begin{split}
(\tilde u-\tilde v)^+(X_t)&=(\tilde u-\tilde v)^+(X_0)-\int_0^t \mathbf1_{\{\tilde u(X_s)>\tilde v(X_s)\}}f(X_s)\,ds 
\\&\quad+\int_0^t \mathbf1_{\{\tilde u(X)_{s-}>\tilde v(X)_{s-}\}}\,dB^x_s
+\int_0^t \mathbf1_{\{\tilde u(X)_{s-}>\tilde v(X)_{s-}\}}\,dA^x_s
\\&\quad+\int_0^t \mathbf1_{\{\tilde u(X)_{s-}>\tilde v(X)_{s-}\}}\,d(M^x-N^x)_s+L^x_t,\quad t\ge 0
\end{split}
\]
$\mathbb P_x\text{-a.s.}$ for some increasing c\`adl\`ag process $L^x$. Let $(\tau_k)$ 
be a localizing sequence for $N^x$ and $M^x$. Replacing $t$ by $\tau_k$ 
and taking the expectation  $\mathbb E_x$ on both sides of the above equality yields 
\[
\begin{split}
(\tilde u-\tilde v)^+(x)&=\mathbb E_x (\tilde u-\tilde v)^+(X_0)\\&\le \mathbb E_x(\tilde u-\tilde v)^+(X_{\tau_k})+
\mathbb E_x\int_0^{\tau_k} \mathbf1_{\{\tilde u(X_s)>\tilde v(X_s)\}}f(X_s)\,ds\\&\le
\mathbb E_xUf^+(X_{\tau_k})+
\mathbb E_x\int_0^{\tau_k} \mathbf1_{\{\tilde u(X_s)>\tilde v(X_s)\}}f(X_s)\,ds.
\end{split}
\]
The result follows by letting $k\to \infty$ and applying Lemma \ref{lm.basic}(3).
\end{proof}

\begin{proposition}
\label{prop.if1}
Assume that $(T_t)$ is sub-Markovian. Let $g\in L^p(E;m)$ and suppose that $u:=Rg$ is nonnegative.
If $\varphi\in C^1_b([0,\infty))$ is convex, then there exists a
supermedian function $v$ such that
\[
\varphi(u)
=
\varphi(0)+R(\varphi'(u)g)-v.
\]
\end{proposition}
\begin{proof}
By Remark \ref{rem.main} we may assume that $E$ is a Lusin
topological space, $\mathcal A=\mathcal B(E)$, and that there exists
a right Markov process $\mathbb X$ representing $(T_t)$ (see Definition \ref{def.rep}).
Set
\[
w^\pm:=Ug^\pm,
\qquad
w:=w^+-w^-.
\]
Then $w=u$ a.e. By Lemma \ref{lm.basic}(3), for a.e. $x\in E$ there
exist uniformly integrable martingales $M^{+,x}$ and $M^{-,x}$,
starting from zero, such that
\[
w^\pm(X_t)
=
w^\pm(X_0)
-
\int_0^t g^\pm(X_s)\,ds
+
M_t^{\pm,x}.
\]
Consequently,
\[
w(X_t)
=
w(X_0)
-
\int_0^t g(X_s)\,ds
+
M_t^x,
\]
where $M^x:=M^{+,x}-M^{-,x}$ is a uniformly integrable martingale.
By the It\^o--Meyer formula \cite[Theorem IV.70]{Protter}, there exist a positive additive
functional $A$ and a uniformly integrable martingale $N^x$, with
$N_0^x=0$, such that
\[
\varphi(w)(X_t)
=
\varphi(w)(X_0)
-
\int_0^t
\varphi'(w(X_s))g(X_s)\,ds
+
A_t+N_t^x,\quad t\ge 0.
\]
Evaluating this identity at $t=\zeta$ and taking expectations yields
\[
\varphi(u(x))
=
\varphi(0)
+
R(\varphi'(u)g)(x)
-
v(x)
\quad\text{a.e.},
\]
where $v(x):=\mathbb E_xA_\zeta$. 
Using the Markov property of $\mathbb X$ and the additivity of $A$, one checks that
$v$ is a supermedian function. This finishes the proof.  
\end{proof}

\section{The method of sub- and supersolutions for sub-Markovian semigroups}
\label{sec5}
We now turn to nonlinear resolvent equations of the form
\begin{equation}
\label{eq.sub}
u=\gamma+Rg(\cdot,u),
\end{equation}
where $\gamma$ is a supermedian function and
$g:E\times\mathbb R\to\mathbb R$ is a Carath\'eodory function. 
We begin by introducing the corresponding notions of sub- and supersolution.

\begin{definition}
We say that a  function $w\in L^p(E;m)$ is a subsolution to
\eqref{eq.sub}
if $g^+(\cdot,w)\in L^p(E;m)$ and there exists a supermedian function $v$
such that
\[
w+Rg^-(\cdot,w)=\gamma+Rg^+(\cdot,w)-v.
\]
\end{definition}

\begin{definition}
We say that a  function $w\in L^p(E;m)$ is a supersolution to
\eqref{eq.sub}
if $g^+(\cdot,w)\in L^p(E;m)$ and there exists a supermedian function $v$
such that
\[
w+Rg^-(\cdot,w)=\gamma+Rg^+(\cdot,w)+v.
\]
\end{definition}

\begin{proposition}
\label{prop.maxsub}
Assume that $(T_t)$ is sub-Markovian. If $u,w$ are  subsolutions to
 \eqref{eq.sub}, then $u\vee w$ is a subsolution to \eqref{eq.sub}.
\end{proposition}
\begin{proof}
Since $u$ is a subsolution  to \eqref{eq.sub} if and only if $u_0:=u-\gamma$
is a subsolution  to \eqref{eq.sub} with $\gamma=0$ and $g$ replaced by $g_0(x,y):=g(x,y+\gamma(x))$,
we may and will assume without loss of generality that $\gamma=0$.

By the definition of a subsolution to \eqref{eq.sub} there exist supermedian functions  $v_1, v_2$ such that
\[
u=R g(\cdot,u)-v_1, \quad w=R g(\cdot,w)-v_2.
\]
By Remark \ref{rem.main} we may assume that $E$ is a Lusin topological space, 
$\mathcal A=\mathcal B(E)$, and  there exists a right Markov process $\mathbb X$  that represents 
the semigroup $(T_t)$ (see Definition \ref{def.rep}).   
We choose $\mathcal B(E)$-measurable representatives of the $m$-equivalence classes of $u$ and $w$
and continue to denote them by  $u$ and $w$. By Lemma \ref{lm.basic}(1)
there exist $m$-versions $\tilde v_1,\tilde v_2$ of $v_1,v_2$, respectively, that are $\mathcal U$-excessive.
Let $\tilde u:= Ug(\cdot,u)-\tilde v_1$ on the set $\{U|g(\cdot,u)|+\tilde v_1<\infty\}$ and zero on its complement, and
$\tilde w:= Ug(\cdot,w)-\tilde v_2$ on the set $\{U|g(\cdot,w)|+\tilde v_2<\infty\}$ and zero on its complement.
Set $E_0:=\{U|g(\cdot,u)|+\tilde v_1+U|g(\cdot,w)|+\tilde v_2<\infty\}$.
By  Lemma \ref{lm.basic} there exist  uniformly integrable  martingales  
$M^{u,x}, N^{w,x}$, local martingales $M^{v_1,x}, N^{v_2,x}$,  
and positive additive functionals  $A,B$ such  that, for every $x\in E_0$,
$$
\begin{gathered}
\tilde u\left(X_t\right)+\tilde v_1(X_t)=\tilde u\left(X_0\right)+\tilde v_1\left(X_0\right)-\int_0^t g\left(X_r,  u\left(X_r\right)\right) dr
+\int_0^tdM^{u,x}_r, \quad t \geq 0,\,\mathbb P_x\text{-a.s.}, \\
\tilde w\left(X_t\right)+\tilde v_2(X_t)=\tilde w\left(X_0\right)+\tilde v_2(X_0)-\int_0^t g\left(X_r,  w\left(X_r\right)\right) d r
+\int_0^t dN^{w,x}_r, \quad t \geq 0,\,\mathbb P_x\text{-a.s.},\\
\tilde v_1\left(X_t\right)=\tilde v_1\left(X_0\right)-A_t+\int_0^tdM^{v_1,x}_r, \quad 
\tilde v_2\left(X_t\right)=\tilde v_2\left(X_0\right)-B_t+\int_0^t dN^{v_2,x}_r, \quad t \geq 0,\,\mathbb P_x\text{-a.s.}
\end{gathered}
$$
Set $M^x:= M^{u,x}-M^{v_1,x}$ and $N^x:= N^{w,x}-N^{v_2,x}$.
Then for $x\in E_0$,
$$
\begin{gathered}
\tilde u\left(X_t\right)=\tilde u\left(X_0\right)-\int_0^t g\left(X_r,  u\left(X_r\right)\right) dr
+A_t+\int_0^tdM^{x}_r, \quad t \geq 0,\,\mathbb P_x\text{-a.s.}, \\
\tilde w\left(X_t\right)=\tilde w\left(X_0\right)-\int_0^t g\left(X_r, w\left(X_r\right)\right) d r
+B_t+\int_0^t dN^{x}_r, \quad t \geq 0,\,\mathbb P_x\text{-a.s.}
\end{gathered}
$$
By the It\^o--Meyer formula (see, e.g., \cite[Theorem IV.70]{Protter})
and Lemma \ref{lm.basic} there exists a  positive additive functional  $C$ such that
$$
\begin{aligned}
(\tilde u \vee \tilde w)\left(X_t\right)= & (\tilde u \vee \tilde w)\left(X_0\right)-\int_0^t \mathbf{1}_{\{\tilde u(X)_{r} \geq \tilde w(X)_{r}\}}
g\left(X_r,  u\left(X_r\right)\right) \,d r+\int_0^t \mathbf{1}_{\{\tilde u(X)_{r-} \geq \tilde w(X)_{r-}\}} dA_r \\
& +\int_0^t \mathbf{1}_{\{\tilde u(X)_{r-} \geq \tilde w(X)_{r-}\}} \,d M^x_r
-\int_0^t\mathbf{1}_{\{\tilde u(X)_{r-} < \tilde w(X)_{r-}\}} g\left(X_r,  w\left(X_r\right)\right) \,d r \\
& +\int_0^t \mathbf{1}_{\{\tilde u(X)_{r-} <\tilde w(X)_{r-}\}} \,d B_r+\int_0^t \mathbf{1}_{\{\tilde u(X)_{r-} 
< \tilde w(X)_{r-}\}} \,d N^x_r+C_t, \quad t\ge 0,\, \mathbb P_x \text {-a.s.}
\end{aligned}
$$
for $x\in E_0$.
By \eqref{eq.equiv123}, $Ug(\cdot,u)=Ug(\cdot,\tilde u)$ a.e. and $Ug(\cdot,w)=Ug(\cdot,\tilde w)$ a.e.  Consequently,
for a.e. $x\in E$,
$$
\begin{aligned}
(\tilde u \vee \tilde w)\left(X_t\right)= & (\tilde u \vee \tilde w)\left(X_0\right)
-\int_0^t g\left(X_r,(\tilde u \vee \tilde w)\left(X_r\right)\right) \,d r
\\& +\int_0^t \mathbf{1}_{\{\tilde u(X)_{r-} \geq \tilde w(X)_{r-}\}} \,d A_r
+\int_0^t \mathbf{1}_{\{\tilde u(X)_{r-} < \tilde w(X)_{r-}\}} \,d B_r +C_t
\\&+\int_0^t \mathbf{1}_{\{\tilde u(X)_{r-} \geq \tilde w(X)_{r-}\}} \,d M^x_r
+\int_0^t \mathbf{1}_{\{\tilde u(X)_{r-} < \tilde w(X)_{r-}\}}\,d N^x_r, \quad t \geq 0, \mathbb P_x\text {-a.s. }
\end{aligned}
$$
 Let $(\tau_k)$ be a localizing sequence for $N^x$ and $M^x$ that consists of exact terminal times
 (take e.g. $\tau_k:=\inf\{t>0: \tilde v_1(X_t)+\tilde v_2(X_t)\ge k\}$, see the proof of 
 \cite[Theorem 51.1]{Sharpe}). Then substituting  $t\mapsto \tau_k$
 and applying $\mathbb E_x$ to both sides of the above equality gives

\begin{equation}
\label{eq.ttlim}
\begin{split}
(\tilde u \vee \tilde w)(x)= & \mathbb{E}_x(\tilde u \vee \tilde w)\left(X_{\tau_k}\right)
+\mathbb{E}_x \int_0^{\tau_k} g\left(X_r,(\tilde u \vee \tilde w)\left(X_r\right)\right) \,d r \\
& -\mathbb{E}_x \int_0^{\tau_k} \mathbf{1}_{\{\tilde u(X)_{r-} \geq \tilde w(X)_{r-}\}} \,d A_r
-\mathbb{E}_x \int_0^{\tau_k} \mathbf{1}_{\{\tilde u(X)_{r-} < \tilde w(X)_{r-}\}} \,d B_r-\mathbb{E}_x C_{\tau_k} \text { a.e. }
\end{split}
\end{equation}
Observe that
\begin{equation}
\label{eq.radon}
\lim_{k\to\infty}\mathbb{E}_x(\tilde u \vee \tilde w)\left(X_{\tau_k}\right)
=-\lim_{k\to\infty}\mathbb{E}_x( \tilde v_1\wedge \tilde v_2)\left(X_{\tau_k}\right).
\end{equation}
Indeed, for $x\in E_0$,
\[
\begin{split}
(\tilde u\vee \tilde w)(x)&=(\tilde u-\tilde w)^+(x)+\tilde w(x)\\&= \left(Ug(\cdot,u)(x)
-Ug(\cdot,w)(x)-\tilde v_1(x)+\tilde v_2(x)\right)^+ +Ug(\cdot,w)(x)-\tilde v_2(x)
\end{split}
\]
On the other hand
\[
-(\tilde v_1\wedge \tilde v_2)(x)=(-\tilde v_1(x))\vee (-\tilde v_2(x))= (-\tilde v_1(x)+\tilde v_2(x))^+-\tilde v_2(x),\quad x\in E.
\]
Consequently, for $x\in E_0$,
\[
\begin{split}
&|(\tilde u\vee \tilde w)(x)+\tilde v_1(x)\wedge \tilde v_2(x)|\\&\quad
\le \left| \left(Ug(\cdot, u)(x)-Ug(\cdot, w)(x)-\tilde v_1(x)+\tilde v_2(x)\right)^+- (-\tilde v_1(x)+\tilde v_2(x))^+ \right|
\\&\quad\quad+U|g(\cdot, w)|(x)
\le U|g(\cdot, u)|(x)+2U|g(\cdot, w)|(x).
\end{split}
\]
Now using the fact that $E_0$ is an absorbing set (see Lemma \ref{lm.basic}(6)) and the 
strong Markov property of $\mathbb X$, we conclude
\[
\mathbb E_x|(\tilde u\vee \tilde w)(X_{\tau_k})+\tilde v_1(X_{\tau_k})\wedge \tilde v_2(X_{\tau_k})|
\le \mathbb E_x\int_{\tau_k}^\infty|g(X_s, u(X_s))|\,ds+2\mathbb E_x\int_{\tau_k}^\infty|g(X_s, w(X_s))|\,ds.
\]
The right-hand side of the above inequality tends to zero whenever the  integrals are finite. 
By the assumptions, this holds for  $m$-a.e. $x\in E$.
From \eqref{eq.radon} and  Lemma \ref{lm.basic}, we infer  that 
$h(x):=-\lim_{k\to\infty}\mathbb{E}_x(\tilde u \vee \tilde w)\left(X_{\tau_k}\right)$
is a supermedian function. Let $\tilde h$ denote its excessive regularization.
Letting $k\to \infty$ in \eqref{eq.ttlim} yields 
$$
\tilde u \vee \tilde w=Ug(\cdot, \tilde  u \vee  \tilde w)-\tilde v, \quad \text { a.e., }
$$
where 
\[
\tilde v(x):=\tilde h(x)+\mathbb{E}_x \int_0^{\infty} \mathbf{1}_{\{\tilde u(X)_{r-} \geq \tilde w(X)_{r-}\}} \,d A_r
+\mathbb{E}_x \int_0^{\infty} \mathbf{1}_{\{\tilde u(X)_{r-} < \tilde w(X)_{r-}\}}\,d B_r+\mathbb{E}_x C_{\infty}. 
\]
It is routine to  check that each term in the definition of $\tilde v$ is supermedian  
(use additivity of $A,B,C$ and the Markov property of $\mathbb X$). 
Therefore, $\tilde v$ is supermedian, which in turn implies that  $u \vee w$ is a subsolution to \eqref{eq.sub}.
\end{proof}

\begin{proposition}
\label{prop.exist}
Suppose that there exist positive $\mathcal A$-measurable functions $h_1,h_2$ on $E$ such that
\begin{equation}
\label{eq.bound1}
-h_1(x)\le g(x,y) \leq  h_2(x), \quad x \in  E,y \in \mathbb{R},
\end{equation}
and assume that $Rh_1\in L^p(E;m)$ and  $h_2\in L^p(E;m)$. Let $\gamma$ be a bounded supermedian function.
Then there exists a solution to \eqref{eq.sub}.
\end{proposition}
\begin{proof}
{\bf Step 1.} Assume additionally that $h_1\in L^p(E;m)$.
Set $r:=\|\gamma\|+\|R(h_1+h_2)\|$. For $u \in L^p(E;m)$, we let
$$
\Phi(u)=\gamma+R g(\cdot, u).
$$
Observe that by the assumptions,
$$
|\Phi(u)| \leq \gamma+ Rh_1+Rh_2.
$$
Hence
$$
\|\Phi(u)\| \leq r
$$
By \eqref{eq.bound1}, kernel representation of the resolvent,  and  the fact that $g$ is a Carath\'eodory function and $h_1,h_2\in L^p(E;m)$, 
the dominated convergence theorem shows that $\Phi$ is continuous.
The compactness of $\Phi(L^p(E;m))$ follows easily from \eqref{eq.bound1}, 
$h_1,h_2\in L^p(E;m)$ and the compactness of the resolvent $(R_\alpha)$.
The existence result follows from Schauder’s fixed-point theorem.

{\bf Step 2}. By  Step 1 for any $n\ge 1$ there exists a solution $u_n\in L^p(E;m)$ to 
\begin{equation}
\label{eq.subnnn}
u=\gamma+Rg_n(\cdot,u),
\end{equation}
where $g_n(x,y):= g(x,y)\vee (-n)$. Note that $Rg^+(\cdot,u_n)$, $Rg^-_n(\cdot,u_n)$ are supermedian functions.
Therefore, by  Deny's theorem (see Theorem  \ref{lm.upseq})  there exists a subsequence, still indexed by $(n)$, 
such that $(Rg^+(\cdot,u_n))$
and $(Rg^-_n(\cdot,u_n))$  converge $m$-a.e., which implies that $u_n$  
converges $m$-a.e. to an $\mathcal A$-measurable  function $u$.
Furthermore, by the right-hand side inequality in \eqref{eq.bound1}, the  dominated convergence theorem  and the compactness 
of the resolvent $(R_\alpha)$, $(g^+(\cdot,u_n))$ and $(Rg^+(\cdot,u_n))$  are  convergent in $L^p(E;m)$
to $(g^+(\cdot,u))$ and $(Rg^+(\cdot,u))$, respectively.
On the other hand,   by \eqref{eq.bound1}, $|u_n|\le \gamma +Rh_1+Rh_2\in L^p(E;m)$.
Thus,  by  the  dominated convergence theorem again, $u_n\to u$ in $L^p(E;m)$.
Now, we want to pass to the limit in 
\begin{equation}
\label{eq.subnnnnn}
u_n+Rg^-_n(\cdot,u_n)=\gamma+Rg^+(\cdot,u_n).
\end{equation}
It remains to prove the convergence of $(Rg^-_n(\cdot,u_n))$. By \eqref{eq.bound1}
\[
R|g^-_n(\cdot,u_n)-g^-(\cdot,u)| \leq 2 R h_1 \in L^p(E;m) .
\]
Hence
\[
\|R g^-_n(\cdot,u_n)-R g^-(\cdot,u)\| \leq\|R|g^-_n(\cdot,u_n)-g^-(\cdot,u)|\| \longrightarrow 0 .
\]
\end{proof}

In what follows, we impose the following integrability assumptions on $g$:

{\bf Assumption (Int):}
\begin{enumerate}
\item[(i)] if $u_1\le u_2$, $u_1,u_2\in L^p(E;m)$, and furthermore 
$Rg^-(\cdot,u_1), Rg^-(\cdot,u_2) \in L^p(E;m)$, $g^+(\cdot,u_1), g^+(\cdot,u_2)\in L^p(E;m)$,
then 
\[
\sup_{u_1\le y\le u_2} g^+(\cdot,y), \,R(\sup_{u_1\le y\le u_2} g^-(\cdot,y))\in L^p(E;m);
\]
\item[(ii)]  if $u\in L^p(E;m)$, then $g^+(\cdot,u)\in L^p(E;m)$ 
and if $(u_n)\subset L^p(E;m)$, $u_n\nearrow u$, then $g^+(\cdot,u_n)\to g^+(\cdot,u)$ in $L^p(E;m)$.
\end{enumerate}

\begin{theorem}
\label{th.im4}
Assume that $(T_t)$ is sub-Markovian and $\gamma$ is a bounded supermedian function. 
\begin{enumerate}
\item[(1)] Assume  (Int). Let $\psi \in L^p(E;m)$. Suppose that there exists a subsolution $\underline{u}$ 
to \eqref{eq.sub} such that $\underline{u} \leq \psi$. Then there exists a maximal subsolution 
$u^*$ to  \eqref{eq.sub}  such that $u^* \leq \psi $.
\item[(2)] Assume (Int), and  suppose that there exists a subsolution $\underline{u}$ and a supersolution $\bar{u}$ 
to  \eqref{eq.sub}  such that $\underline{u} \leq \bar{u}$. 
Then there exists a maximal solution $u$ to  \eqref{eq.sub}  with  $\underline{u} \leq u \leq \bar{u}$. 
Moreover, this solution is also the  maximal subsolution lying below $\bar{u}$.
\end{enumerate}
\end{theorem}
\begin{proof}
Proof of  (1). Set
$$
S_\psi=\{w: w \text { is a subsolution to  \eqref{eq.sub}  and } w \leq \psi\}
$$
By the assumptions, $S_\psi$ is nonempty and the number
$$
\alpha:=\sup _{w\in S_{\psi}} \int_E |w|^{p-1}w
$$
is finite. 
Let $(w_n) \subset S_\psi$ be such that $w_1=\underline{u}$ and  $ \int_E |w_n|^{p-1}w_n  \nearrow \alpha$. Set

$$
u_n:=\max \left\{w_1, \ldots, w_n\right\}, \quad n \geq 1 .
$$
By Proposition \ref{prop.maxsub}, $\left\{u_n\right\} \subset S_\psi$. Set $u^*=\sup _{n \geq 1} u_n$. 
Clearly, $\underline{u} \leq u^* \leq \psi$.  By the assumptions made on $\psi$ and (Int), 
we have $u^*\in L^p(E;m)$ and  $g^+\left(\cdot, u_n\right)\to g^+\left(\cdot, u^*\right)$  in $ L^p(E;m)$. 
Since $u_n$ is a subsolution to \eqref{eq.sub} there exists a supermedian $v_n \in L^p(E;m)$ such that
\begin{equation}
\label{eq.eqsubn}
u_n=\gamma+R g\left(\cdot, u_n\right)-v_n.
\end{equation}
Let $r_0(x,dy)$ be the kernel  representing operator $R$ (see \eqref{eq.rochlin})
and $N$ be the set of $x\in E$ for which $u_n(x)$ does not converge to $u^*(x)$.
We have $0=(R\mathbf1_N)(x)=r_0(x,N)$ a.e.,  so  $u_n$
converges to $u^*$ $r_0(x,\cdot)$-a.e. for a.e. $x\in E$. 
We further  observe that
\[
R g^-(\cdot,u^*)\in L^p(E;m).
\]
Indeed, by \eqref{eq.eqsubn} and monotonicity of the sequence $(u_n)$,  
\[
R g^-(\cdot,u_n)
\le
\gamma+R g^+(\cdot,u_n)+u_1^-.
\]
By \textup{(Int)(ii)} and Fatou's lemma,
\[
\|R g^-(\cdot,u^*)\|_{L^p(E;m)}^p
\le
\liminf_{n\to\infty}
\|R g^-(\cdot,u_n)\|_{L^p(E;m)}^p
<\infty.
\]
Thus $R g^-(\cdot,u^*)\in L^p(E;m)$.
Since also $g^+(\cdot,u^*)\in L^p(E;m)$, condition
\textup{(Int)(i)} may now be applied to $u_1$ and $u^*$.
Hence, setting
\[
h(x):=\sup_{u_1(x)\le y\le u^*(x)}g^-(x,y),
\]
we obtain $Rh\in L^p(E;m)$. Clearly, $|g^-(\cdot,u_n)|\le h$. Thus, by the kernel representation of the resolvent,
and the the dominated convergence theorem,
\[
\|R|g^-(\cdot,u_n)-g^-(\cdot,u^*)|\|\to 0.
\]
This, together with the convergence of $(g^+(\cdot,u_n))$, implies that $v_n$ converges in $L^1(E;m)$.
Since $\sup_{n\ge 1}\|v_n\|<\infty$,  its limit $v$ is supermedian.
Passing to the limit in \eqref{eq.eqsubn} yields
$u^* \in S_\psi$. It remains to prove that  $u^*$ is maximal. 
Let $w \in S_\psi$. Clearly $w \vee u_n \nearrow w \vee u^*$. 
By Proposition \ref{prop.maxsub}, $w \vee u_n \in S_\psi$. Thus,
\[
\alpha=\lim _{n \rightarrow \infty} \int_E |w_n|^{p-1}w_n  \leq \lim _{n \rightarrow \infty} \int_E |u_n|^{p-1}u_n  
\leq \lim _{n \rightarrow \infty}\int_E |w \vee u_n|^{p-1}(w \vee u_n)  \leq \alpha.
\]
Therefore,
\[
\int_E \left(|w \vee u^*|^{p-1}(w \vee u^*)-|u^*|^{p-1}u^* \right)=0
\]
Hence, $u^* = w \vee u^*$, which implies that $w \leq u^*$.

Proof of  (2). By Remark \ref{rem.main} we may assume that $E$ is a Lusin topological space, 
$\mathcal A=\mathcal B(E)$, and  there exists a right Markov process $\mathbb X$  that represents 
the semigroup $(T_t)$ (see Definition \ref{def.rep}). Set
\[
\hat{g}(x,y)=g(x,(y \wedge \bar{u}(x)) \vee \underline{u}(x)), \quad x \in  E,\, y \in \mathbb{R} .
\]
By condition (i) of (Int), $\hat{g}$ satisfies  all the assumptions of  Proposition \ref{prop.exist}. 
Therefore there exists a solution $\hat{u}$ to \eqref{eq.sub} with $g$ replaced by $\hat{g}$. 
Since $\underline{u}$ is a subsolution to  \eqref{eq.sub}, there exists a supermedian function  $v$ such that
\[
\underline{u}=\gamma+R g(\cdot, \underline{u})-v.
\]
We choose $\mathcal B(E)$-measurable representatives of the $m$-equivalence classes of
$\underline{u}$, $\hat u$ and  $\bar u$ and continue to denote them by the same symbols. By Lemma \ref{lm.basic}(1)
there exist $m$-versions $\tilde \gamma,\tilde v$ of $\gamma,v$, respectively, that are $\mathcal U$-excessive.
Let $\tilde{\underline{u}}:= \tilde \gamma+Ug(\cdot,\underline{u})-\tilde v$ 
on the set $\{\tilde\gamma+U|g(\cdot,\underline{u})|+\tilde v<\infty\}$ 
and zero on its complement, and
$\hat u':= \tilde\gamma+U\hat g(\cdot,\hat u)$ on the set 
$\{\tilde\gamma+U|\hat g(\cdot,\hat u)|<\infty\}$ and zero on its complement.
Set $E_0:=\{\tilde\gamma+U|\hat g(\cdot,\hat u)|+\tilde v+U|g(\cdot,\underline{u})|<\infty\}$.
By Lemma \ref{lm.basic},  the It\^o--Meyer formula 
and   arguments analogous to those leading to \eqref{eq.ttlim} (cf. the proof of Proposition \ref{prop.maxsub})
we obtain that there exists a positive additive functional $(A_t)$ such that 
$$
\begin{aligned}
(\underline{\tilde u}(x)-\hat{u}'(x))^{+} &
\le \mathbb{E}_x\left(\underline{\tilde u}\left(X_{\tau_k}\right)-\hat{u}'\left(X_{\tau_k}\right)\right)^{+}\\&
+\mathbb{E}_x \int_0^{\tau_k} \mathbf{1}_{\{\underline{\tilde u}(X)_{r-}>\hat{u}'(X)_{r-}\}}
\left(g\left(X_r, \underline{\tilde u}\left(X_r\right)\right)-\hat{g}\left(X_r, \hat{u}'\left(X_r\right)\right)\right) \,d r \\
& -\mathbb{E}_x \int_0^{\tau_k} \mathbf{1}_{\{\underline{\tilde u}(X)_{r-}>\hat{u}'(X)_{r-}\}} \,d A_r \quad \text { a.e. }
\end{aligned}
$$
for an increasing sequence $(\tau_k)$ of stopping times such that $\tau_k\to\infty$. 
Observe that, by the definition of $\hat{g}$, we have 
$\mathbf{1}_{\{\underline{\tilde u}>\hat{u}'\}}(g(\cdot, \underline{\tilde u})-\hat{g}(\cdot, \hat{u}')) \leq 0$. Thus,
$$
(\underline{\tilde u}(x)-\hat{u}'(x))^{+} \leq 
\mathbb{E}_x\left(\underline{\tilde u}\left(X_{\tau_k}\right)-\hat{u}'\left(X_{\tau_k}\right)\right)^{+} 
\quad \text { a.e. }
$$
Letting $k \rightarrow \infty$  yields
$$
(\underline{u}(x)-\hat{u}(x))^{+} \leq 0\quad\text{a.e.,}
$$
and so $\underline{u} \leq \hat{u} $. We used here the fact that 
\[
(\underline{\tilde u}(x)-\hat{u}'(x))^{+} \le (U|g(\cdot,\underline u)-\hat g(\cdot,\hat u)|)(x),\quad x\in E_0,
\]
which combined with the strong Markov property of $\mathbb X$
 and the fact that $E_0$ is absorbing (see Lemma \ref{lm.basic}(6)) gives 
\[
\mathbb{E}_x\left(\underline{\tilde u}\left(X_{\tau_k}\right)-\hat{u}'\left(X_{\tau_k}\right)\right)^{+} 
\le \mathbb E_x\int_{\tau_k}^\infty|g(X_s,\underline u(X_s))-\hat g(X_s,\hat u(X_s))|\,ds\to 0,\quad k\to \infty \quad x\in E_0.
\]
Analogous reasoning for $\hat{u}, \bar{u}$ shows that $\hat{u} \leq \bar{u}$.
 Consequently, $\hat{g}(\cdot, \hat{u})=g(\cdot, \hat{u})$. 
 Therefore, $\hat{u}$ is, in fact, a solution to \eqref{eq.sub} and $\underline{u} \leq \hat{u} \leq \bar{u}$.
 Now, we shall show the existence of a maximal solution to \eqref{eq.sub} lying between $\underline{u}$ and  $\bar{u}$. 
 Applying (1) with $\psi=\bar{u}$ gives the existence of a maximal subsolution $u^*$ to \eqref{eq.sub}
 such that $u^* \leq \bar{u}$. Since $\underline u$ is a subsolution to
 \eqref{eq.sub}, we clearly have $\underline{u} \leq u^* \leq \bar{u}$.  By what has already been  proved, 
 there exists a solution $\hat{u}$ to \eqref{eq.sub} such that $u^* \leq \hat{u} \leq \bar{u}$.
 On the other hand, since $\hat{u}$ is also  a subsolution to \eqref{eq.sub}, $\hat{u} \leq u^*$.
Thus $u^*$ is a maximal solution to \eqref{eq.sub}.
\end{proof}

\begin{proposition}  
\label{prop4.2}
Assume that $(T_t)$ is sub-Markovian and $g$ satisfies (Int). 
Let $\gamma,\underline\gamma$ be  bounded and supermedian functions. 
Let $\psi \in L^p(E;m)$ be  a supersolution to \eqref{eq.sub} 
and  $u$ be a maximal   solution to \eqref{eq.sub}  such that $u \le \psi$.
Let  $\underline{g}:E\times\mathbb R\to\mathbb R$ be a Carath\'eodory function and $\underline{u}$ be a positive solution to 
\[
w=\underline\gamma+R\underline g(\cdot,w)
\]
such that $ \underline {u} \leq \psi$.
If $\gamma-\underline \gamma$ is a supermedian function, 
and  $\underline{g}(\cdot, \underline{u}) \leq g(\cdot, \underline{u})$,  then $\underline{u} \leq u$.
\end{proposition}
\begin{proof}
Observe that 
\[
\underline u=\gamma+Rg(\cdot,\underline u)-(\gamma-\underline\gamma)
-(Rg(\cdot,\underline u)-R\underline g(\cdot,\underline u)).
\]
By the assumptions, $\gamma-\underline\gamma$ is a supermedian function, and  
$g(\cdot, \underline{u})-\underline{g}(\cdot, \underline{u})\ge 0$. 
Thus $g^-(\cdot, \underline{u})\le \underline{g}^-(\cdot, \underline{u})$, and so 
$Rg^-(\cdot, \underline{u})\in L^p(E;m)$. By (Int), $g^+(\cdot, \underline{u}) \in L^p(E;m)$. 
Consequently, $R[g(\cdot, \underline{u})-\underline{g}(\cdot, \underline{u})]$
is a supermedian function and  $\underline{u}$ is a subsolution to \eqref{eq.sub}.
By Theorem \ref{th.im4}(2), $u$ is a maximal subsolution to \eqref{eq.sub} such that $u\le \psi$.
Indeed, otherwise there exists a subsolution $w\le\psi$ of \eqref{eq.sub} such that $m(\{w>u\})>0$. 
By Proposition \ref{prop.maxsub}, the function $\hat w:=u\vee w$ is a subsolution of \eqref{eq.sub}
and satisfies $u<\hat w\le\psi$. Hence, by Theorem \ref{th.im4}, there exists a solution $\hat u$ such that $\hat w\le\hat u\le\psi$, 
contradicting the maximality of $u$. Consequently, since $u$ is a maximal subsolution to \eqref{eq.sub}, we obtain
that   $\underline{u} \leq u$.
\end{proof}

\section{Main result:  sub-Markovian case}
\label{sec6}

\begin{definition}
\label{def.main}
 We call $u\in L^p(E;m)$ a solution to \eqref{eq1.1} if 
\[
u+Rf^-(\cdot,u)=Rf^+(\cdot,u).
\]
\end{definition}
This definition is tailored to the fact that (F1) guarantees $f^+(\cdot,u)\in L^p(E;m)$ and hence $Rf^+(\cdot,u)\in L^p(E;m)$, 
whereas  $f^-(\cdot,u)$ need not belong to $L^p(E;m)$ in the present general setting.
Let us recall that conditions (F1)--(F3) as well as the notion of  $\lambda_1(a_0)$,  $\lambda_1(a_\infty)$ have been 
introduced in Section \ref{sec.prelim}.

\begin{proposition}
\label{prop.shifted-equation}
Let $\gamma$ be a bounded supermedian function,
$\kappa\geq 0$, 
\[
L_\kappa:=L-\kappa I,
\qquad
f^\kappa(x,y):=f(x,y)+\kappa y,
\]
and    $u\in L^p(E;m)$ be a positive function. Then  $u$ is a solution to
\begin{equation}
\label{eq.res123}
u=\gamma+Rf(\cdot,u)
\end{equation}
if and only if it is a solution to
\begin{equation}
\label{eq.res1234}
u=\gamma-\kappa R_\kappa\gamma+R_\kappa f^\kappa(\cdot,u).
\end{equation}
\end{proposition}

\begin{proof}
The resolvent identity gives
\begin{equation}\label{eq.resolvent-shift-forms}
R_\kappa+\kappa R_\kappa R=R,
\qquad
R_\kappa+\kappa RR_\kappa=R.
\end{equation}
Assume first that $u$ satisfies \eqref{eq.res123}. 
Using 
\eqref{eq.resolvent-shift-forms}, we obtain
\[
\begin{split}
R_\kappa f^\kappa(\cdot,u)&=R_\kappa f(\cdot,u)+\kappa R_\kappa u=R_\kappa f(\cdot,u)
+\kappa R_\kappa Rf(\cdot,u)+\kappa R_\kappa\gamma
\\&=\bigl(R_\kappa+\kappa R_\kappa R\bigr)f(\cdot,u)
+\kappa R_\kappa\gamma=Rf(\cdot,u)+\kappa R_\kappa\gamma=u-\gamma+\kappa R_\kappa\gamma.
\end{split}
\]
So $u$ satisfies \eqref{eq.res1234}. Conversely, assume that $u$ satisfies \eqref{eq.res1234}. 
Using 
\eqref{eq.resolvent-shift-forms}, we get
\[
\begin{split}
Rf(\cdot,u)
&=R(f^\kappa(\cdot,u)-\kappa u)=Rf^\kappa(\cdot,u)-\kappa Ru=
\bigl(R_\kappa+\kappa RR_\kappa\bigr)f^\kappa(\cdot,u)-\kappa Ru\\&
=u-\gamma+\kappa R_\kappa\gamma+\kappa R(u-\gamma+\kappa R_\kappa\gamma)-\kappa Ru=u-\gamma.
\end{split}
\]
Hence $u$ satisfies  \eqref{eq.res123}.
\end{proof}

\begin{theorem}
\label{th.mainm}
Assume that $(T_t)$ is sub-Markovian and (F1)-(F2) and (F3)$_1$. If 
\[
\lambda_1(a_0)<0< \lambda_1(a_\infty),
\]
then there exists a strictly positive   solution to \eqref{eq1.1}.
\end{theorem}
\begin{proof}
Fix $\delta>0$. By (F3)$_1$, Proposition \ref{prop.shifted-equation} and Remark \ref{rem.shift1}, 
we may assume without loss of generality that
$f(x,y)>0$ for $x\in E$ and $y\in(0,\delta]$, and $R1\le 1$.  Fix $l\ge 1$ such  that $\lambda_1(a_0^l)<0<\lambda_1(a_\infty^l)$. 

{\bf Step 1.} Let us define 
\begin{equation}
\label{eq.fkn}
f_{k,n}(x,y):= \min\{f^+(x,y)\wedge k, ky\}-f^-(x,y)\wedge n.
\end{equation}
Observe that 
\begin{enumerate}
\item[(i)] $f^-_{k,n}(x,y)=0,\, x\in E,\, y\in [0,\delta]$;
\item[(ii)] $f^-_{k,n}(x,y)\le n,\, x\in E,\, y\ge 0$ and $f^+_{k,n}(x,y)\le k,\, x\in E,\, y\ge 0$;
\item[(iii)] $f^+_{k,n}(x,y)/y\le k,\, x\in E,\, y> 0$;
\item[(iv)]  $f^-_{k,n}(x,y)/y\le \delta^{-1} f^-_{k,n}(x,y),\, x\in E,\, y> 0$.
\end{enumerate}
Set 
\begin{equation}
\label{eq.hn}
h_n(y):=\frac{1}{n(1+y)},\quad y\ge 0.
\end{equation}
Observe that $h_n$ is decreasing and convex on $[0,\infty)$. Furthermore, $h_n(0)=1/n$.
Let $\phi_1$ be a principal eigenfunction of $-L$ with the principal eigenvalue $\lambda_1>0$.
Then, by Proposition \ref{prop.if1}, there exists a supermedian function $v_n$ such that
\[
h_n(\phi_1)=\frac1n+\lambda_1R(h'_n(\phi_1)\phi_1)-v_n.
\]
Hence
\[
h_n(\phi_1)=\frac1n+R(f_{k,n}(\cdot,h_n(\phi_1)))+\left[ \lambda_1R(h'_n(\phi_1)\phi_1)-v_n-R(f_{k,n}(\cdot,h_n(\phi_1)))\right].
\]
Since $h_n(\phi_1)\le \frac1n\le \delta$ for $n\ge n_0:=
1+ [1/\delta]$, we obtain, by using (i),  that for such $n$ the function $h_n(\phi_1)$
is a strictly positive subsolution to
\begin{equation}
\label{eq.pkn}
u=\frac1n+R(f_{k,n}(\cdot,u)).
\end{equation}
Using (ii), one readily verifies that the  function $\psi_k:= 1+Rk$ is a supersolution to the above equation. 
Clearly,   $h_n(\phi_1)\le \psi_k$. Consequently, by Theorem \ref{th.im4}, there exists a maximal solution
$u_{k,n}$ to \eqref{eq.pkn}   lying between $h_n(\phi_1)$ and $\psi_k$. 
Clearly, in view of Theorem  \ref{th.im4}, $u_{k,n}$ is a maximal solution to  \eqref{eq.pkn}  such that $u_{k,n}\le \psi_k$.
Observe that
\begin{equation}
\label{eq4.comp11}
 u_{k,n}\ge u_{k,n+1},\quad u_{k,n}\le u_{k+1,n},\quad k\ge 1,\, n\ge n_0.
\end{equation}
Indeed, applying Proposition \ref{prop4.2} with $(\underline u,\underline g,\underline \gamma)= (u_{k,n+1},f_{k,n+1},1/(n+1))$,
 $(u,g,\gamma)= (u_{k,n},f_{k,n},1/n)$ and $\psi:=\psi_k$ gives the left-hand side inequality in \eqref{eq4.comp11}, while 
applying  the same proposition with 
$(\underline u,\underline g,\underline \gamma)= (u_{k,n},f_{k,n},1/n)$,
 $(u,g,\gamma)= (u_{k+1,n},f_{k+1,n},1/n)$ and $\psi=\psi_{k+1}$ gives the right-hand side inequality. 
In particular 
\[
\|u_{k,n}\|_\infty\le 1+k,\quad n\ge n_0,
\]
and  $(u_{k,n})_{n\ge n_0}$ converges in $L^p(E;m)$.
Using  (F3) and (ii) we deduce that $u_k:=\lim_{n\to \infty} u_{k,n}$ satisfies
\[
u_k=R(f_{k}(\cdot,u_k)),\quad k\ge 1,
\]
where $f_{k}(x,u):= \min\{f^+(x,u)\wedge k, ku\}-f^-(x,u).$ 

{\bf Step 2.} We  prove that $u_k$ is strictly positive for $k\ge l$. 
Fix  $\theta>0$
such that $a^l_0+\lambda_1(a_0^l)+\theta\ge 0$ in $E$.
Suppose that $c_n:=\|u_{k,n}\|\searrow 0$.
Then, letting $w_{k,n}:= u_{k,n}/c_n$, we get
\[
w_{k,n}=\frac{1}{nc_n}+R(f_{k,n}(\cdot,u_{k,n}))/c_n
\]
By (iii) and (iv),
\[
|f_{k,n}(x,u_{k,n}(x))/c_n|\le |w_{k,n}(x)|(k+\delta^{-1}\sup_{0\le y\le k+1}f^-(x,y)),
\]
which, combined with the compactness of the resolvent  and the fact that $\|w_{k,n}\|=1$,
implies that, along a  subsequence, $(R(f_{k,n}(\cdot,u_{k,n}))/c_n)_{n\ge n_0}$ converges in $L^p(E;m)$.
This and boundedness of the sequence $(\|w_{k,n}\|)_{n\ge n_0}$ yield boundedness of the sequence 
$(\frac{1}{nc_n})_{n\ge n_0}$.
As a result, 
we deduce that $(w_{k,n})_{n\ge n_0}$, along a  subsequence, converges in $L^p(E;m)$
to a function that we denote by $w_k$.  By Proposition \ref{prop.shifted-equation} 
\[
w_{k,n}=\frac{1}{nc_n}-\theta R_\theta\frac{1}{nc_n} +R_\theta(f_{k,n}(\cdot,u_{k,n})+\theta u_{k,n})/c_n.
\]
Let $\varphi_1'$ be a dual principal eigenfunction of $-L-a_0^l$. Then, for $k\ge l$,
\[
\begin{split}
\left\langle w_{k,n},a_0^l\varphi_1'+\lambda_1(a_0^l)\varphi_1'+\theta\varphi_1'\right\rangle&
=\left\langle \frac{1}{nc_n}-\theta R_\theta\frac{1}{nc_n}, a_0^l\varphi_1'+\lambda_1(a_0^l)\varphi_1'
+\theta\varphi_1'\right\rangle \\&\quad
+\left\langle R_\theta(f_{k,n}(\cdot,u_{k,n})+\theta u_{k,n})/c_n, a_0^l\varphi_1'
+\lambda_1(a_0^l)\varphi_1'+\theta \varphi_1'\right\rangle 
\\&\ge \left\langle R_\theta(f_{k,n}(\cdot,u_{k,n})+\theta u_{k,n})/c_n, a_0^l\varphi_1'+\lambda_1(a_0^l)\varphi_1'
+\theta \varphi_1'\right\rangle\\&
=\left\langle f_{k,n}(\cdot,u_{k,n})/c_n+\theta w_{k,n}, \varphi_1'\right\rangle \\&=
\left\langle w_{k,n}f_{k,n}(\cdot,u_{k,n})/u_{k,n}+\theta w_{k,n}, \varphi_1'\right\rangle.
\end{split}
\]
Applying $\liminf_{n\to \infty}$ to both  sides of the above inequality and using the definition of $a_0^l$ yields
\[
\langle w_{k},a_0^l\varphi_1'+\lambda_1(a_0^l)\varphi_1'
+\theta \varphi_1'\rangle\ge \langle w_{k}a_0^l+\theta w_k, \varphi_1'\rangle,
\]
which implies that $\lambda_1(a_0^l) \langle w_{k},\varphi_1'\rangle\ge0$, a contradiction.

Now suppose that $c_n\searrow c>0$. Then $u_k$ is nontrivial. We shall show, by using the positivity-improving
 property of the resolvent,
that in this case $u_k$ must be  strictly positive for any $k\ge l$. Aiming for a contradiction, suppose that 
$m(\{u_k=0\})>0$.
Let 
\[
\rho(x):=x\wedge \delta,\quad \rho_n(x):= n\int_0^{1/n} \rho(x+y)\,dy,\, n\ge 1.
\]
Note that  $\rho_n\in C^1_b([0,\infty))$ is  positive, non-decreasing and concave
on $[0,\infty)$. Furthermore,  $\rho_n(x)\searrow \rho(x)$ and $\mathbf1_{[0,\delta-1/n]}(x)\le \rho_n'(x)\le \mathbf1_{[0,\delta]}(x)$.
By Proposition \ref{prop.if1} (applied to $-\rho_n$) there exists a supermedian function $v_n$ such that 
\[
\rho_n(u_k)=\rho_n(0)+R(\rho_n'(u_k)f_k(\cdot,u_k))+v_n.
\]
Since
\[
\{\rho_n'\ne0\}\subset[0,\delta],
\qquad
f_k(x,y)\ge0
\quad\text{for }y\in[0,\delta],
\]
it follows that $\rho_n(u_k)$ is supermedian for every $n\ge1$. Passing to the limit yields that $u_k\wedge\delta$ is supermedian. 
Hence
\[
R_1(u_k\wedge\delta)\le u_k\wedge\delta.
\]
Since  $u_k\wedge\delta\not\equiv0$, the positivity-improving property of the resolvent gives
\[
R_1(u_k\wedge\delta)\gg0,
\]
and therefore $u_k\wedge\delta\gg0$, contrary to
$m(\{u_k=0\})>0$. This proves that $u_k$ is strictly positive.

{\bf Step 3. Letting $k\to \infty$.} 
We first note   that by \eqref{eq4.comp11},
\begin{equation}
\label{eq4.comp12}
u_{k}\le u_{k+1},\quad k\ge l.
\end{equation}
Let $\varphi'_1$ be a dual principal eigenfunction of  $-L-a_\infty^l$. 
Suppose that $\sup_{k\ge 1}\|u_k\|<\infty$. Then $u:=\lim_{k\to \infty} u_k$
is the desired solution. To see this, let $M>0$.  By Kato's inequality (Proposition \ref{lm2}) 
\[
0\le (u_k-M)^+\le R(\mathbf1_{\{u_k>M\}}f_k(\cdot,u_k)).
\]
Consequently,
\[
R(\mathbf1_{\{u_k>M\}}f^-_k(\cdot,u_k))\le R(\mathbf1_{\{u_k>M\}}f^+_k(\cdot,u_k)).
\]
Thus, by (F1),
\[
\begin{split}
\theta(M,\eta)&:=\limsup_{k\to \infty}\int_{\{u_k> M\}} f^-_k(\cdot,u_k) R^*\eta
\\&\le \limsup_{k\to \infty} \int_{\{u_k>M\}} f^+_k(\cdot,u_k) R^*\eta\le  \int_{\{u> M\}} (h+\lambda u) R^*\eta.
\end{split}
\]
As a result, for any positive $\eta\in L^{p'}(E;m)$, $\lim_{M\to\infty} \theta(M,\eta)=0$. 
Let us rewrite the equation for $u_k$ in the following form
\[
\int_{ E} u_k\eta+\int_{\{u_k\le M\}} f^-_k(\cdot,u_k) R^*\eta+\int_{\{u_k> M\}} f^-_k(\cdot,u_k) R^*\eta
=\int_{ E} f^+_k(\cdot,u_k) R^*\eta.
\]
We now take   $\limsup_{k\to \infty}$ on both sides of the equality. Then applying the dominated convergence theorem
to the second term on the left-hand side (with the aid of (F3)) and the term on the right-hand side (with the aid of (F1))   yields
\[
\int_{ E} u\eta+\int_{\{u\le M\}} f^-(\cdot,u) R^*\eta+\theta(M,\eta)=\int_{ E} f^+(\cdot,u) R^*\eta
\]
for any positive $\eta\in L^{p'}(E;m)$.   Now, letting $M\to \infty$ gives that $u$ is a solution to \eqref{eq1.1}.

Suppose now that $\sup_{k\ge 1}\|u_k\|=\infty$.
Set $c_k:= \|u_k\|$  and $w_k:= c_k^{-1} u_k$.
There exists a subsequence such that  $(w_k)$  converges a.e. 
Indeed, 
\[
w_k=R[f^+_k(\cdot,u_k)/c_k]-R[f^-_k(\cdot,u_k)/c_k],
\]
and both terms on the right-hand side are supermedian. Consequently, by Deny's theorem (see Theorem  \ref{lm.upseq}),
both sequences $(R[f^+_k(\cdot,u_k)/c_k])_{k\ge 1}$, $(R[f^-_k(\cdot,u_k)/c_k])_{k\ge 1}$  
converge  a.e. after passing to a subsequence.
Furthermore, 
\[
R[f^+_k(\cdot,u_k)/c_k]\le  R[h/c_k+\lambda u_k/c_k]\le  Rh+\lambda Rw_k,\quad c_k\ge 1,
\]
and this  implies  that the limit of $(R[f^+_k(\cdot,u_k)/c_k])_{k\ge 1}$ is finite a.e., 
which in turn allows us to conclude that $(w_k)$ converges a.e. We denote its limit by $w$.
Since
\[
w_k\le R[f^+_k(\cdot,u_k)/c_k]\le  Rh+\lambda Rw_k,\quad c_k\ge 1,
\]
we conclude, by using compactness of the resolvent, that $(w_k^p)$ is uniformly integrable.
Thus $w_k\to w$ in $L^p(E;m)$. Hence $\|w\|=1$, and consequently $w$ is nontrivial.
We have 
\[
\langle w_k,-L^*\varphi'_1\rangle=\langle c_k^{-1}f_k(\cdot,u_k),\varphi_1'\rangle.
\]
Let  $A:= \{\sup_{k} u_k=\infty\}$. 
By  (F1) and the convergence of $(w_k)$ the sequence 
$(c_k^{-1}f^+_k(\cdot,u_k)\varphi_1')$ is uniformly integrable and dominates the sequence $(c_k^{-1}f_k(\cdot,u_k)\varphi_1')$. 
Therefore, we may apply Fatou's lemma to the right-hand side of the above equality and obtain
\[
\begin{split}
\limsup_{k\to \infty}\langle c_k^{-1}f_k(\cdot,u_k),\varphi_1'\rangle  
& \le 
\int_E \limsup_{k\to \infty}\left(c_k^{-1}f_k(\cdot,u_k)\varphi_1' \right)
\\&=\int_E \limsup_{k\to \infty} \left(w_k f_k(\cdot,u_k)/u_k\right)\varphi_1' 
\\&
\le \langle \mathbf1_Aw a^l_\infty,\varphi_1'\rangle \le \langle w a_\infty^l,\varphi_1'\rangle
\end{split}
\]
since $A^c\subset \{w=0\}$. Therefore
\[
\langle w,a_\infty^l\varphi'_1+\lambda_1(a_\infty^l)\varphi'_1\rangle
=\langle w,-L^*\varphi'_1\rangle\le \langle w a_\infty^l,\varphi_1'\rangle.
\]
As a result
\[
\lambda_1(a_\infty^l)\langle w,\varphi'_1\rangle\le 0,
\]
which implies, since $\varphi_1'$ is strictly positive and $w$ is positive nontrivial, that $\lambda_1(a_\infty^l)\le 0$, a contradiction.
\end{proof}

\begin{remark}
\label{rem.ind1}
Let us point out that the solution constructed in the proof of
Theorem~\ref{th.mainm} is not, a priori, independent of the shift used
at the beginning of the proof. Indeed, any shift
$(L_\kappa,f^\kappa)$, in the notation of
Proposition~\ref{prop.shifted-equation}, with $\kappa\ge C_\delta\vee1$
may be used. The proof then proceeds with $R$ replaced by $R_\kappa$
and $f$ replaced by $f^\kappa$. In particular, the approximating
solution $u_{k,n}^{(\kappa)}$ satisfies
\begin{equation}
\label{eq.pkndep}
u
=
\frac1n+
R_\kappa\left(
f^\kappa_{k,n}(\cdot,u)
\right).
\end{equation}
Equivalently, by the resolvent identity, it solves 
\[
u=\frac1n+R\left(f^\kappa_{k,n}(\cdot,u)-\kappa u+\frac{\kappa}{n}\right).
\]
Thus the approximating problems themselves depend on $\kappa$.
Consequently, the construction given in the proof does not by itself
show that the limiting solution is independent of the chosen shift.
\end{remark}

\begin{lemma}
\label{lm.shiftmon}
Let $1\le \kappa_1\le \kappa_2$ be two admissible shifts used in the proof of
Theorem~\ref{th.mainm}. For $i=1,2$, let
\[
u_k^{(\kappa_i)}
:=
\lim_{n\to\infty}u_{k,n}^{(\kappa_i)},
\]
where $u_{k,n}^{(\kappa_i)}$ is the maximal solution constructed in
Step~1 of the proof of Theorem~\ref{th.mainm} with $R$ replaced by
$R_{\kappa_i}$ and $f$ replaced by
\[
f^{\kappa_i}(x,y):=f(x,y)+\kappa_i y.
\]
Then, for every $k\ge1$,
\[
u_k^{(\kappa_2)}\le u_k^{(\kappa_1)}.
\]
\end{lemma}

\begin{proof}
Fix $\kappa\ge1$. Let  $f_{k,n}^\kappa$ and $f_k^\kappa$ be defined as in \eqref{eq.fkn} 
with $f$ replaced by $f^\kappa$. By Step~1 of the proof of Theorem~\ref{th.mainm}, we may 
consider a maximal solution $u_{k,n}^{(\kappa)}$,
in the order interval $[h_n(\phi_1),1+R_\kappa k]$, to 
\[
u=\frac1n+R_\kappa f_{k,n}^\kappa(\cdot,u),
\]
and set $u_k^{(\kappa)}:= \inf_{n\ge 1} u_{k,n}^{(\kappa)}$. Then $u_k^{(\kappa)}$ is a  positive solution of
\begin{equation}
\label{eq.ukshift}
u=R_\kappa f_k^\kappa(\cdot,u).
\end{equation}
We first observe that $u_k^{(\kappa)}$ is the maximal subsolution of
\eqref{eq.ukshift} below $\psi_k^{(\kappa)}:= 1+R_\kappa k$.
Indeed, let $w\le\psi_k^{(\kappa)}$ be a positive subsolution of
\eqref{eq.ukshift}. Thus there exists a supermedian function $v$ such
that
\[
w=R_\kappa f_k^\kappa(\cdot,w)-v.
\]
Since $f_{k,n}^\kappa\ge f_k^\kappa$, we may write
\[
w=\frac1n+R_\kappa f_{k,n}^\kappa(\cdot,w)-\left[v+\frac1n+R_\kappa\bigl(f_{k,n}^\kappa(\cdot,w)-f_k^\kappa(\cdot,w)\bigr)\right].
\]
The expression in square brackets is supermedian. Hence $w$ is a subsolution of
\[
u=\frac1n+R_\kappa f_{k,n}^\kappa(\cdot,u).
\]
By Theorem~\ref{th.im4}, $u_{k,n}^{(\kappa)}$ is the maximal
subsolution of this equation below $\psi_k^{(\kappa)}$. Therefore
\[
w\le u_{k,n}^{(\kappa)}
\]
for every $n$, and letting $n\to\infty$ gives $w\le u_k^{(\kappa)}$.
We now compare the two shifts. 
For $y\ge0$, set
\[
\eta_k(y):=k\wedge ky,
\]
and observe that 
\[
f_k^\kappa(x,y)
=
\bigl(f(x,y)+\kappa y\bigr)\wedge \eta_k(y).
\]
Consequently,
\begin{equation}
\label{eq.shiftorder}
\begin{split}
f_k^{\kappa_2}(x,y)-(\kappa_2-\kappa_1)y
&=
\bigl(f(x,y)+\kappa_1y+(\kappa_2-\kappa_1)y\bigr)\wedge \eta_k(y)-(\kappa_2-\kappa_1)y\\
&=
\bigl(f(x,y)+\kappa_1y\bigr)
\wedge
\bigl(\eta_k(y)-(\kappa_2-\kappa_1)y\bigr)
\\&\le
\bigl(f(x,y)+\kappa_1y\bigr)\wedge \eta_k(y)
=
f_k^{\kappa_1}(x,y).
\end{split}
\end{equation}
By \eqref{eq.ukshift},
\[
u_k^{(\kappa_2)}=R_{\kappa_2}f_k^{\kappa_2}(\cdot,u_k^{(\kappa_2)}).
\]
Proposition~\ref{prop.shifted-equation}, applied with $\gamma=0$,
gives
\begin{equation}
\label{eq.shiftback}
u_k^{(\kappa_2)}=R_{\kappa_1}\left(f_k^{\kappa_2}(\cdot,u_k^{(\kappa_2)})-(\kappa_2-\kappa_1)u_k^{(\kappa_2)}\right).
\end{equation}
By \eqref{eq.shiftorder},
\[
g:=f_k^{\kappa_1}(\cdot,u_k^{(\kappa_2)})-f_k^{\kappa_2}(\cdot,u_k^{(\kappa_2)})+(\kappa_2-\kappa_1)u_k^{(\kappa_2)}\ge0.
\]
From   \eqref{eq.shiftback}, we obtain 
\[
u_k^{(\kappa_2)}=R_{\kappa_1}f_k^{\kappa_1}(\cdot,u_k^{(\kappa_2)})-R_{\kappa_1}g.
\]
Thus $u_k^{(\kappa_2)}$ is a subsolution of
\[
u=R_{\kappa_1}f_k^{\kappa_1}(\cdot,u),
\]
and $u_k^{(\kappa_2)}\le R_{\kappa_1}k\le  \psi^{(\kappa_1)}_k$.
Since, as shown above, $u_k^{(\kappa_1)}$ is the maximal subsolution
of this equation  below $\psi^{(\kappa_1)}_k$, it follows that
\[
u_k^{(\kappa_2)}\le u_k^{(\kappa_1)}.
\]
This proves the assertion.
\end{proof}

\section{Main result: the general case}
\label{sec7}

We now remove the sub-Markovian assumption from Theorem \ref{th.mainm}. 
The reduction to the sub-Markovian case is achieved by means of a Doob transform associated 
with  a  strictly positive supermedian function $\psi$ in (F3).
Define
\begin{equation}
\label{eq.semdt}
T^{\psi}_t(f):=\frac{T_t(\psi f)}{\psi},
\end{equation}
for $f\in L^p(E; m_\psi)$, where
\[
m_\psi:=\psi^p\cdot m.
\]
The family $(T_t^{\psi})$ is a $C_0$-semigroup on $L^p(E;m_\psi)$ since, by using the isometry
$$
J_\psi: L^p(E ; m_\psi) \rightarrow L^p(E;m), \quad J_\psi u:=\psi u,
$$
we may represent it as 
$$
T_t^\psi=J_\psi^{-1} T_t J_\psi.
$$
From this one easily sees that  its resolvent is  compact and positivity-improving, and its 
generator $(L^{\psi},\mathscr D(L^{\psi}))$ is given by
\begin{equation}
\label{eq.opdt}
\mathscr D(L^{\psi})=\{u\in L^p(E;m_\psi): u\psi\in \mathscr D(L)\},\quad L^{\psi}u= L(\psi u)/\psi,\quad u\in \mathscr D(L^{\psi}).
\end{equation}
Moreover,
\[
T_t^{\psi}1=T_t\psi/\psi\le1,
\]
so $(T_t^{\psi})$ is sub-Markovian.

\begin{remark}
\label{rem.doobgen}
The Doob transform also allows one to extend the Deny-type
compactness theorem of Section~\ref{sec3}, as well as the results of
Sections~\ref{sec4} and \ref{sec5}, beyond the sub-Markovian setting.
Indeed, the map
\[
u\longmapsto \frac{u}{\phi_1},
\]
where $\phi_1$ is a principal eigenfunction of $-L$, 
is an order-preserving isometry from $L^p(E;m)$ onto
$L^p(E;m_{\phi_1})$ and maps supermedian functions for $(T_t)$ onto
supermedian functions for $(T_t^{\phi_1})$.
\end{remark}

In what follows, for a given strictly positive supermedian function $\psi$ we introduce  the class  $\mathscr  F_\psi$
of all Carath\'eodory functions $f$ satisfying (F1),(F2),(F3)$_\psi$, and \eqref{eq.nine}, and we let 
\[
I_\psi:\mathscr F_\psi\longrightarrow\mathscr F_1,
\qquad
I_\psi(f)(x,y):=\frac{f(x,\psi(x)y)}{\psi(x)}.
\]

\begin{theorem}
\label{th.main}
Assume that (F1)-(F3) hold. If 
\begin{equation}
\label{eq.nine}
\lambda_1(a_0)<0< \lambda_1(a_\infty),
\end{equation}
then there exists a  strictly positive solution to \eqref{eq1.1}.
\end{theorem}
\begin{proof}
Observe that if $v\in L^p(E;m_\psi)$ solves 
\begin{equation}
\label{eq.mainphi}
-L^{\psi}v=I_\psi(f)(\cdot,v),\quad v\gg 0,
\end{equation}
then $u:= \psi v$ solves \eqref{eq1.1}.
Consequently, to obtain the existence result  for \eqref{eq1.1}, it remains to verify that the pair   $((T^{\psi}_t), I_\psi(f))$ 
satisfies all the assumptions of Theorem \ref{th.mainm}.

First, one easily verifies that $I_\psi(f)$ satisfies (F1),(F2),(F3)$_1$ with the same constants and $h$ 
replaced by $h/\psi$ (which belongs to $L^p(E;m_\psi)$).
For every $k\ge1$, the isometry $J_\psi u=\psi u$ gives
\[
J_\psi(-L^\psi-a_0^k(I_\psi(f)))J_\psi^{-1}=-L-a_0^k(I_\psi(f)),
\quad
J_\psi(-L^\psi-a_\infty^k(I_\psi(f)))J_\psi^{-1}=-L-a_\infty^k(I_\psi(f)).
\]
Hence
\[
\lambda_1(-L^\psi-a_0^k(I_\psi(f)))=\lambda_1(a_0^k(I_\psi(f))),
\qquad
\lambda_1(-L^\psi-a_\infty^k(I_\psi(f)))=\lambda_1(a_\infty^k(I_\psi(f))).
\]
Next, observe that $ a^k_0=a^k_0(I_\psi(f))$ and $ a^k_\infty=a^k_\infty(I_\psi(f))$.
Taking respectively the infimum and the supremum over $k$ yields
\[
\lambda_1^{L^\psi}(a_0(I_\psi(f)))=\lambda_1(a_0),
\qquad
\lambda_1^{L^\psi}(a_\infty(I_\psi(f)))=\lambda_1(a_\infty),
\]
which in turn implies that
\[
\lambda_1^{L^\psi}(a_0(I_\psi(f)))<0<\lambda_1^{L^\psi}(a_\infty(I_\psi(f))).
\]
Consequently, by Theorem \ref{th.mainm} there exists a solution to \eqref{eq.mainphi} and hence a solution to \eqref{eq1.1}.
\end{proof}

\begin{corollary}
Let $\psi$ be a strictly positive supermedian function and $\mathscr  F_\psi$
denote the set of all Carath\'eodory functions $f$ satisfying (F1),(F2),(F3)$_\psi$, and \eqref{eq.nine}.
There exists a map $\mathcal S_\psi: \mathscr F_\psi\to L^p(E;m)$ such that
\begin{enumerate}
\item[(i)] for any $f\in\mathscr F_\psi$ the function  $\mathcal S_\psi(f)$ is a strictly positive solution to \eqref{eq1.1};
\item[(ii)] if $f_1,f_2\in\mathscr F_\psi$ and $f_1\le f_2$,  then $\mathcal S_\psi(f_1)\le \mathcal S_\psi(f_2)$.
\end{enumerate}
\end{corollary}
\begin{proof}
As shown in the proof of Theorem~\ref{th.main}, $I_\psi$ is
order preserving and establishes a correspondence between solutions
of the original and the transformed equations. Thus it suffices to
construct an order-preserving selection map $\mathcal S_1$ for the
transformed, sub-Markovian problem. Once this is done, we set
\[
\mathcal S_\psi(f)
:=
\psi\,\mathcal S_1(I_\psi(f)).
\]

We therefore assume that $(T_t)$ is sub-Markovian and that
\textup{(F1)}, \textup{(F2)}, and \textup{(F3)}$_1$ hold.
Fix $\delta>0$ and, for $f\in\mathscr F_1$, set
\[
C_\delta(f)
:=
\inf\left\{
C\ge0:
f(x,y)\ge-Cy,\quad x\in E,\quad 0\le y\le\delta
\right\},
\]
and
\[
\kappa_f:=C_\delta(f)\vee1.
\]
We define $\mathcal S_1(f)$ to be the solution obtained by the
construction in the proof of Theorem~\ref{th.mainm}, applied to the
shifted pair $(L_{\kappa_f},f^{\kappa_f})$. Indeed,
\[
f^{\kappa_f}(x,y)\ge0,
\qquad x\in E,\quad 0\le y\le\delta,
\]
and, since the shifted semigroup is sub-Markovian and
$\kappa_f\ge1$,
\[
R_{\kappa_f}1\le1.
\]
Let now $f,\hat f\in\mathscr F_1$ satisfy $f\le\hat f$. Then $C_\delta(f)\ge C_\delta(\hat f)$, hence $\kappa_f\ge\kappa_{\hat f}$.
For  $g\in \mathscr F_1$ and $\kappa\ge 1$, denote by
$u_{k,n}^{g,\kappa}$ the maximal solution of 
\[
u=\frac1n+
R_{\kappa}
\bigl(g_{k,n}^{\kappa}(\cdot,u)\bigr)
\]
in the order interval $[h_n(\phi_1),1+R_\kappa k]$ (see Step 1 of the proof of Theorem~\ref{th.mainm}),
and set
\[
u_k^{g,\kappa}:=\lim_{n\to\infty}u_{k,n}^{g,\kappa}.
\]
Since $f\le\hat f$, we have $f_{k,n}^{\kappa_f}\le\hat f_{k,n}^{\kappa_f}$.
Proposition~\ref{prop4.2} consequently yields
\[
u_{k,n}^{f,\kappa_f}
\le
u_{k,n}^{\hat f,\kappa_f}.
\]
Letting $n\to\infty$, we obtain
\[
u_k^{f,\kappa_f}
\le
u_k^{\hat f,\kappa_f}.
\]
On the other hand, since $\kappa_{\hat f}\le\kappa_f$,
Lemma~\ref{lm.shiftmon}, applied to the nonlinearity $\hat f$, gives
\[
u_k^{\hat f,\kappa_f}
\le
u_k^{\hat f,\kappa_{\hat f}}.
\]
Consequently,
\[
u_k^{f,\kappa_f}
\le
u_k^{\hat f,\kappa_{\hat f}}
\]
for every $k$. Letting $k\to\infty$ and using the construction in
Theorem~\ref{th.mainm}, we conclude that
\[
\mathcal S_1(f)\le\mathcal S_1(\hat f).
\]
Thus $\mathcal S_1$ is order preserving.
Finally, since both $I_\psi$ and multiplication by the strictly
positive function $\psi$ preserve order,
\[
\mathcal S_\psi(f)
=
\psi\,\mathcal S_1(I_\psi(f))
\]
has properties \textup{(i)} and \textup{(ii)}.
\end{proof}

\begin{remark}
It is an elementary check that if $\psi_1$, $\psi_2$ are supermedian functions, then $\psi_1\wedge\psi_2$
is a supermedian function as well. As a result, if $f\in\mathscr F_{\psi_1}$ and $\hat f\in \mathscr F_{\psi_2}$, then
$f,\hat f\in \mathscr F_{\psi_1\wedge\psi_2}$, and if further $f\le\hat f$, then 
$\mathcal S_{\psi_1\wedge\psi_2}(f)\le \mathcal S_{\psi_1\wedge\psi_2}(\hat f)$.
\end{remark}

\begin{corollary}
\label{cor.weakerF3}
Assume that \textup{(F1)--(F2)} hold and suppose, in addition, that the following  condition \textup{(F3')} holds:
\begin{enumerate}
\item[{\rm (F3')}] there exists a strictly positive supermedian function $\psi$ such that:
\begin{enumerate}
\item[(i)] there exist $\delta>0$ and $C_\delta\ge0$ such that
\[
f^-(x,y)\le C_\delta y,
\qquad x\in E,\quad 0\le y\le \psi(x)\delta;
\]
\item[(ii)] for every $M>0$, the function 
\[
F_M(x):=\sup_{0\le y\le \psi(x) M}f^-(x,y)
\]
belongs to $L^p(E;m)$.
\end{enumerate} 
\end{enumerate}
If
\begin{equation}
\label{eq.weakerF3.spectral}
\lambda_1(a_0)<0<\lambda_1(a_\infty),
\end{equation}
then \eqref{eq1.1} admits a strictly positive solution.
\end{corollary}

\begin{proof}
As in the proof of Theorem \ref{th.main} we may reduce to the case where $(T_t)$ is sub-Markovian and
$\psi\equiv 1$ in (F3'). 
By Proposition \ref{prop.shifted-equation}, we may assume without
loss of generality that
\begin{equation}
\label{eq.localpos}
R1\le 1,\quad \text{and}\quad f(x,y)\ge0,
\qquad x\in E,\quad 0\le y\le\delta.
\end{equation}
Choose $l\ge1$ so large that
\begin{equation}
\label{eq.choose.l}
\lambda_1(a_0^l)<0<\lambda_1(a_\infty^l).
\end{equation}
For $k,n\ge l$, define, for $x\in E$, $y\ge 0$,
\begin{equation}
\label{eq.newfkn}
f_k(x,y)
:=
\min\{f^+(x,y),ky,k\}-f^-(x,y),\quad  f_{k,n}(x,y)
:=f^+_k(x,y)-\min\{f^-(x,y),ny\}.
\end{equation}
Then
\begin{equation}
\label{eq.estmt}
f_{k,n}^+(x,y)\le ky\wedge k,
\qquad
f_{k,n}^-(x,y)\le ny.
\end{equation}
In particular, $f_{k,n}$ satisfies \textup{(F1)--(F3)}.
Moreover, by \eqref{eq.localpos},
\[
a_0(f_{k,n})=a_0\wedge k=a_0^k,\quad a_\infty(f_{k,n})=\bigl(a_\infty\vee(-n)\bigr)\wedge0.
\]
Consequently, for $k,n\ge l$,
\[
\lambda_1(a_0(f_{k,n}))
=
\lambda_1(a_0^k)
\le \lambda_1(a_0^l)<0,
\]
while, by monotonicity of the generalized principal eigenvalue,
\[
\lambda_1(a_\infty(f_{k,n}))
\ge
\lambda_1(a_\infty^n)
\ge
\lambda_1(a_\infty^l)>0.
\]
Thus $f_{k,n}\in\mathscr F_1$.
Set
\[
u_{k,n}:=\mathcal S_1(f_{k,n}),
\qquad k,n\ge l.
\]
Since $f_{k,n}$ is decreasing in $n$ and increasing in $k$, the
monotonicity of $\mathcal S_1$ gives
\begin{equation}
\label{eq.newmon}
u_{k,n}\ge u_{k,n+1},
\qquad
u_{k,n}\le u_{k+1,n},
\qquad k,n\ge l.
\end{equation}

{\bf Step 1. Letting $n\to\infty$.}  For fixed $k$, define
\[
u_k:=\inf_{n\ge l}u_{k,n}.
\]
Since
\[
0\le u_{k,n}\le Rk\le k,
\]
we have
\begin{equation}
\label{eq.untoUK}
u_{k,n}\longrightarrow u_k
\quad\text{a.e. and in }L^p(E;m).
\end{equation}
We first show that $u_k$ is nontrivial. Suppose, to the contrary,
that $u_k=0$. Then
\[
c_n:=\|u_{k,n}\|\searrow0.
\]
Set $w_{k,n}:=u_{k,n}/c_n$. Then 
\[
w_{k,n}
=
R\left(f_{k,n}(\cdot,u_{k,n})/c_n
\right).
\]
Since $f_{k,n}^-(\cdot,u_{k,n})\le n u_{k,n}$ and
$f_{k,n}^+(\cdot,u_{k,n})\le k u_{k,n} $, both parts belong to
$L^p(E;m)$.
Put
\[
\eta_n
:=
R\left(
f_{k,n}^+(\cdot,u_{k,n})/c_n
\right),
\qquad
\zeta_n
:=
R\left(
f_{k,n}^-(\cdot,u_{k,n})/c_n
\right).
\]
Then
\[
0\le \eta_n\le kRw_{k,n},
\qquad
w_{k,n}=\eta_n-\zeta_n.
\]
Since $\|w_{k,n}\|=1$, the sequences $(\eta_n)$ and $(\zeta_n)$ are
bounded in $L^p(E;m)$. Both consist of supermedian functions.
Hence, by Theorem \ref{lm.upseq}, after passing to a subsequence,
$\eta_n$ and $\zeta_n$ converge a.e. Consequently,
\[
w_{k,n}\longrightarrow w_k
\qquad\text{a.e.}
\]
along the same subsequence.
Furthermore,
\[
0\le w_{k,n}\le \eta_n\le kRw_{k,n}.
\]
By compactness of $R$, after passing to a further subsequence,
$(Rw_{k,n})$ converges in $L^p(E;m)$. It follows that
$(w_{k,n}^p)$ is uniformly integrable. Therefore, by Vitali's
theorem,
\begin{equation}
\label{eq.wnstrong}
w_{k,n}\longrightarrow w_k
\quad\text{in }L^p(E;m).
\end{equation}
In particular, $\|w_k\|=1$.
Let $\varphi_l'$ be a dual principal eigenfunction associated with
$-L-a_0^l$, and choose $\theta>0$ so large that
\[
a_0^l+\lambda_1(a_0^l)+\theta\ge0
\qquad\text{on }E.
\]
By Proposition \ref{prop.shifted-equation},
\[
w_{k,n}
=
R_\theta \left [(f_{k,n}(\cdot,u_{k,n})+\theta u_{k,n})/c_n\right].
\]
Fix $r>0$. By Proposition \ref{lm2}
\[
(w_{k,n}-r)^+
\le
R_\theta
\left(
\mathbf1_{\{w_{k,n}>r\}}(f_{k,n}(\cdot,u_{k,n})+\theta u_{k,n})/c_n
\right).
\]
Consequently,
\begin{equation}
\label{eq.truncwn}
w_{k,n}\wedge r
\ge
R_\theta
\left(
\mathbf1_{\{w_{k,n}\le r\}}(f_{k,n}(\cdot,u_{k,n})+\theta u_{k,n})/c_n
\right).
\end{equation}
Pairing  \eqref{eq.truncwn} with the non-negative function $\bigl(a_0^l+\lambda_1(a_0^l)+\theta\bigr)\varphi_l'$ gives
\begin{equation}
\label{eq.truncpair}
\left\langle w_{k,n}\wedge r,\bigl(a_0^l+\lambda_1(a_0^l)+\theta\bigr)\varphi_l'\right\rangle
\ge\left\langle\mathbf1_{\{w_{k,n}\le r\}}w_{k,n}\left(f_{k,n}(\cdot,u_{k,n})/u_{k,n}+\theta\right),\varphi_l'\right\rangle.
\end{equation}
Since, under the contradiction assumption, $u_{k,n}\to0$ a.e. as $n\to \infty$, we have,  by the definition of $a_0$,
\[
\liminf_{n\to\infty}
\frac{f_{k,n}(x,u_{k,n}(x))}{u_{k,n}(x)}
\ge
a_0(x)\wedge k
\ge a_0^l(x).
\]
Observe that
\[
\{w_{k,n}\le r\}
=
\{u_{k,n}\le r c_n\}.
\]
Since $c_n\to0$, for every fixed $r$ we have
$r c_n\le\delta$ for all sufficiently large $n$. Hence, by
\eqref{eq.localpos}, the integrand on the right-hand side of
\eqref{eq.truncpair} is nonnegative for all sufficiently large $n$.
Choose $r>0$ such that
\[
m(\{w_k=r\})=0.
\]
Using \eqref{eq.wnstrong}, Fatou's lemma, and
\eqref{eq.truncpair}, we obtain
\begin{equation}
\label{eq.limit.r}
\left\langle
w_k\wedge r,
\bigl(a_0^l+\lambda_1(a_0^l)+\theta\bigr)\varphi_l'
\right\rangle\ge\left\langle\mathbf1_{\{w_k<r\}}w_k(a_0^l+\theta),\varphi_l'\right\rangle.
\end{equation}
Taking a sequence $r_j\uparrow\infty$ such that
$m(\{w_k=r_j\})=0$ and letting $j\to\infty$ in
\eqref{eq.limit.r}, we get
\[
\left\langle
w_k,
\bigl(a_0^l+\lambda_1(a_0^l)+\theta\bigr)\varphi_l'
\right\rangle
\ge
\left\langle
w_k(a_0^l+\theta),
\varphi_l'
\right\rangle .
\]
Thus
\[
\lambda_1(a_0^l)
\langle w_k,\varphi_l'\rangle
\ge0,
\]
a contradiction.
Hence $u_k$ is nontrivial.  
We next pass to the limit as $n\to\infty$ in the equation for
$u_{k,n}$. From \eqref{eq.estmt}--\eqref{eq.untoUK}
and the dominated convergence theorem we infer that
\begin{equation}
\label{eq.posnlim}
f_{k,n}^+(\cdot,u_{k,n})=f_k^+(\cdot,u_{k,n})
\longrightarrow
f_k^+(\cdot,u_k)
\quad\text{in }L^p(E;m)
\end{equation}
(see \eqref{eq.newfkn}). It remains to treat the negative part. Fix $M>0$ such that $m(\{u_k=M\})=0$. On
$\{u_{k,n}\le M\}$,
\[
f_{k,n}^-(\cdot,u_{k,n})
\le
F_M,
\]
and \textup{(F3')(ii)} together with dominated convergence yields
\begin{equation}
\label{eq.localq}
\mathbf1_{\{u_{k,n}\le M\}}f_{k,n}^-(\cdot,u_{k,n})
\longrightarrow
\mathbf1_{\{u_k\le M\}}f^-(\cdot,u_k)
\quad\text{in }L^p(E;m).
\end{equation}
On the other hand, Proposition \ref{lm2} gives
\[
(u_{k,n}-M)^+
+
R\left(
\mathbf1_{\{u_{k,n}>M\}}f_{k,n}^-(\cdot,u_{k,n})
\right)
\le
R\left(
\mathbf1_{\{u_{k,n}>M\}}
f_k^+(\cdot,u_{k,n})
\right).
\]
Since $f_k^+(\cdot,u_{k,n})\le k u_{k,n}$ and
$u_{k,n}\le u_{k,l}$,
\begin{equation}
\label{eq.tailn}
R\left(
\mathbf1_{\{u_{k,n}>M\}}f_{k,n}^-(\cdot,u_{k,n})
\right)
\le
kR\left(
\mathbf1_{\{u_{k,l}>M\}}u_{k,l}
\right).
\end{equation}
The right-hand side converges to zero in $L^p(E;m)$ as
$M\to\infty$, uniformly in $n$.
Using the kernel representation of $R$ and Fatou's lemma, we also
obtain
\[
R\left(
\mathbf1_{\{u_k>M\}}f^-(\cdot,u_k)
\right)
\le
kR\left(
\mathbf1_{\{u_{k,l}>M\}}u_{k,l}
\right).
\]
Combining this estimate with \eqref{eq.localq} and
\eqref{eq.tailn}, first letting $n\to\infty$ and then
$M\to\infty$, yields
\begin{equation}
\label{eq.negpotlim}
Rf_{k,n}^-(\cdot,u_{k,n})
\longrightarrow
Rf^-(\cdot,u_k)
\quad\text{in }L^p(E;m).
\end{equation}
Passing to the limit in
\[
u_{k,n}+Rf_{k,n}^-(\cdot,u_{k,n})
=
Rf_k^+(\cdot,u_{k,n})
\]
and using \eqref{eq.untoUK}, \eqref{eq.posnlim}, and
\eqref{eq.negpotlim}, we obtain
\begin{equation}
\label{eq.uk}
u_k+Rf^-(\cdot,u_k)=Rf_k^+(\cdot,u_k).
\end{equation}
We now show that $u_k$ is strictly positive. Since
$f_k(x,y)\ge0$ for $0\le y\le\delta$, the localized concave
truncation argument with the function $\rho_n$ used in Step~2 of the proof of
Theorem \ref{th.mainm} 
applies to \eqref{eq.uk} and shows that $u_k\gg0$.

{\bf Step 2. Letting $k\to\infty$.}  By \eqref{eq.newmon},
\[
u_k\le u_{k+1},
\qquad k\ge l.
\]
Suppose first that $\sup_{k\ge l}\|u_k\|<\infty$. Then
\[
u_k\uparrow u
\quad\text{a.e. and in }L^p(E;m)
\]
for some $u\in L^p(E;m)$. Since every $u_k$ is strictly positive,
$u\gg0$. By \textup{(F1)},
\[
f_k^+(\cdot,u_k)\le f^+(\cdot,u_k)\le h+\lambda u,
\]
and therefore
\[
f_k^+(\cdot,u_k)\longrightarrow f^+(\cdot,u)\quad\text{in }L^p(E;m).
\]
For the negative part we argue as above. For every $M>0$ such that
$m(\{u=M\})=0$, \textup{(F3')(ii)} gives
\[
\mathbf1_{\{u_k\le M\}}f^-(\cdot,u_k)\longrightarrow\mathbf1_{\{u\le M\}}f^-(\cdot,u)\quad\text{in }L^p(E;m).
\]
Moreover, Proposition \ref{lm2} and \textup{(F1)} imply
\[
R\left(\mathbf1_{\{u_k>M\}}f^-(\cdot,u_k)\right)\le 
R\left(\mathbf1_{\{u_k>M\}}f_k^+(\cdot,u_k)\right)\le R\left(\mathbf1_{\{u>M\}}(h+\lambda u)\right).
\]
The last term converges to zero in $L^p(E;m)$ as $M\to\infty$.
Using again the kernel representation and Fatou's lemma for the
limit function, we conclude that
\[
Rf^-(\cdot,u_k)\longrightarrow Rf^-(\cdot,u) \quad\text{in }L^p(E;m).
\]
Passing to the limit in \eqref{eq.uk} gives
\[
u+Rf^-(\cdot,u)=Rf^+(\cdot,u),
\]
so $u$ is a strictly positive solution of \eqref{eq1.1}.
It remains to exclude the possibility that
\[
\sup_{k\ge l}\|u_k\|=\infty.
\]
The argument is identical to that used in Step~3 of the proof of
Theorem~\ref{th.mainm}. Therefore the unbounded alternative cannot occur,
and the proof is complete.
\end{proof}

\section{Necessity and uniqueness under strict monotonicity}
\label{sec8}
Consider a Lusin topological space $Y$  and a bounded Borel measure $\mu$  on $Y$.
Let $(S_t)$ be a sub-Markovian $C_0$-semigroup on  $L^p(Y;\mu)$ with generator $(A,\mathscr D(A))$.
 Assume that  the associated  resolvent  $(G_\alpha)$ is compact and positivity improving and that $s(A)<0$.
We write $G:=G_0$.
Suppose further that  there exists a  right Markov process $\mathbb X$ that represents $(S_t)$.
\begin{lemma}
\label{lm.lm.lm}
For  any  $f\in p\mathcal B(Y)$, with $\int_Y|f|^p\,d\mu<\infty$, 
and any $t>0$,
\[
S_tf(x)=\mathbb E_xf(X_t)\quad\mu\text{-a.e.}
\]
In particular,  if $f,g\in p\mathcal B(Y)$
are equal $\mu$-a.e., then $P_tf=P_tg$ $\mu$-a.e.
(see the convention adopted in Remark  \ref{rem.not}).
\end{lemma}
\begin{proof}
Suppose that $f=U_\alpha g$ for some bounded $g\in p\mathcal B(Y)$.
Then, by Lemma \ref{lm.basic}, $\mathbb E_xf(X_t)\to \mathbb E_xf(X_s),\, t\searrow s$ for any $x\in Y$ and $s\ge 0$.
This continuity, combined with \eqref{eq.equiv123} and well-known properties of the 
Laplace  transform, implies that for any $t\ge 0$
\begin{equation}
\label{eq.ulel}
S_tf(x)=\mathbb E_xf(X_t)\quad\mu\text{-a.e.}
\end{equation}
Now, let $f\in p\mathcal B(Y)$. Then, using once again \eqref{eq.equiv123}, we obtain that there exists 
an increasing  sequence $(\alpha_n)$ tending to infinity such that $f_n:=\alpha_n U_{\alpha_n} f\to f$ $\mu$-a.e. and 
in $L^p(Y;\mu)$. By \eqref{eq.ulel}
\[
S_tf_n(x)=\mathbb E_xf_n(X_t)\quad\mu\text{-a.e.}
\]
Consequently, by continuity of $S_t$ and the  dominated convergence theorem, 
we infer that the asserted equality holds for bounded functions in $p\mathcal B(Y)$. 
The general result is then obtained easily by using truncations. 
Finally, let $\varphi_1'$ be a dual principal eigenfunction  associated with   $A$.  Suppose that 
$f,g\in p\mathcal B(Y)$ are equal $\mu$-a.e. Then, by what has been already proved,
\[
\int_Y (P_t|f-g|)\varphi'_1\,d\mu=\int_Y (S_t|f-g|)\varphi'_1\,d\mu=\int_Y e^{-\lambda_1(-A)t}|f-g|\varphi'_1\,d\mu=0.
\]
Thus, $P_tf=P_tg$ $\mu$-a.e. 
\end{proof}

As before, the transition function of $\mathbb X$ allows us to extend each $G_\alpha$ to nonnegative measurable 
functions; throughout this section, this extension is denoted again by $G_\alpha$.
Let  $V:Y\to \mathbb R$ be a   Borel measurable function such that 
\begin{equation}
\label{eq.c1}
V^-\in \mathcal B_b(Y),\quad \mathbb P_x\left(\int_0^t V(X_s)\,ds<\infty,\,\text{ for every } t<\zeta\right)
=1\quad \mu\text{-a.e. }x\in Y.
\end{equation}
By Lemma \ref{lm.lm.lm} and \cite[Theorem 3.14, Lemma 3.15]{Getoor} the following family  of operators can be regarded, 
with the convention adopted in Remark \ref{rem.not}, as a strongly continuous semigroup on $L^p(Y;\mu)$,
\begin{equation}
\label{eq.fkf1}
P^V_tf(x):=\mathbb E_x\left[e^{-\int_0^t V(X_s)\,ds}f(X_t)\right].
\end{equation}
This means  that whenever $f,\tilde f\in p\mathcal B(Y)$ and $f=\tilde f$ $\mu$-a.e., then $P^V_tf=P^V_t\tilde f$ $\mu$-a.e.
Furthermore,   if for $f\in L^p(Y;\mu)$  we define $T^V_tf$ as the $\mu$-equivalence class of the function $g$ 
given by   $g:=P^V_t \tilde f^+-P^V_t \tilde f^-$ on the set $\{x\in Y: P^V_t \tilde f^+(x)+P^V_t \tilde f^-(x)<\infty\}$
and zero otherwise for a  representative $\tilde f$ of the  $\mu$-equivalence class  $f$, then $(T^V_t)$ 
is a   strongly continuous  semigroup  on $L^p(Y;\mu)$.
We  denote by $A-V$ the operator generated by the semigroup  $(T^V_t)$   on  $L^p(Y;\mu)$, and by $(G^V_\alpha)$ its resolvent.

\begin{theorem}
\label{th.pert1}
Assume \eqref{eq.c1}. Then the following statements hold.
\begin{enumerate}
\item[(i)] $(G^V_\alpha)$ is positivity improving, and 
  for any $\alpha>s(A-V)$ the operator  $(G^V_\alpha)^2$ is compact.
\item[(ii)] For any $\alpha>s(A-V)$, the spectral radius of $G^V_\alpha$ (denoted by $r(G^V_\alpha)$) satisfies:
$r(G^V_\alpha)>0$ and $r(G^V_\alpha)$ is the unique eigenvalue of $G^V_\alpha$
with strictly positive eigenfunction.
\item[(iii)] For any $\alpha>s(A-V)$, 
\[
r(G^V_\alpha)=\frac{1}{\alpha-s(A-V)}.
\]
\item[(iv)]  If $V_1\le V_2$ satisfy the same conditions as $V$, then $s(A-V_1)\ge s(A-V_2)$. 
If additionally $\mu(\{V_1<V_2\})>0$, then $s(A-V_1)>s(A-V_2)$.
\item[(v)] If $V_k:=V\wedge k$, $k\geq1$, then
\[
s(A-V_k)\downarrow s(A-V)
\qquad\text{as }k\to\infty.
\]
\end{enumerate}
\end{theorem}
\begin{proof}
(i) The compactness of   $(G^V_\alpha)^2$  follows from \cite[Corollary 3.7.15]{MeyerNieberg1991}, 
while irreducibility is a consequence of the 
probabilistic interpretation of the semigroup. (ii) follows from \cite[Theorem~V.5.2(iv) and the Corollary on p.~329]{Schaefer1974}. 
For (iii) see, e.g., \cite[(1.8) in Chapter VI]{EngelNagel2000}. (iv) follows from \cite[Theorem 2.1]{AD} and (iii).
To prove (v), set $V_k:=V\wedge k$, and fix $j\geq1$ and  $\alpha>s(A-V_j)$.
By monotonicity of the sequence $(V_k)$
and 
the  Feynman--Kac formula \eqref{eq.fkf1}, we have
\begin{equation}\label{eq.strong-resolvent-truncation}
G_\alpha^{V_k}f\longrightarrow G_\alpha^Vf
\quad\text{in }L^p(Y;\mu)
\end{equation}
for every $f\in L^p(Y;\mu)$. Furthermore,
\begin{equation}\label{eq.order-resolvent-truncation}
0\leq G_\alpha^V
\leq G_\alpha^{V_{k+1}}
\leq G_\alpha^{V_k},\quad k\ge j.
\end{equation}
Set, for $k\ge j$,
\[
r_k:=r(G_\alpha^{V_k}),
\quad
r:=r(G_\alpha^V).
\]
By \eqref{eq.order-resolvent-truncation}, $r_k\downarrow\rho$
for some $\rho\geq r$. By (ii), $r>0$, and for every $k$ there
exists $u_k\gg0$ such that
\[
\|u_k\|=1,
\qquad
G_\alpha^{V_k} u_k=r_ku_k.
\]
Since $G_\alpha^{V_k}\leq G_\alpha^{V_j}$ for $k\ge j$, positivity gives
\[
(G_\alpha^{V_k})^2\leq (G_\alpha^{V_j})^2.
\]
Therefore,
\begin{equation}\label{eq.uk-domination}
0\leq r^2u_k
\leq r_k^2u_k
=
(G_\alpha^{V_k})^2u_k
\leq (G_\alpha^{V_j})^2u_k.
\end{equation}
We claim that $(u_k)$ has a weakly convergent subsequence. If
$1<p<\infty$, this follows from the reflexivity of $L^p(Y;\mu)$.
If $p=1$, the compactness of $(G_\alpha^{V_j})^2$ implies that the family
$\{(G_\alpha^{V_j})^2u_k:k\geq1\}$
is uniformly integrable. By \eqref{eq.uk-domination}, the sequence
$(u_k)$ is uniformly integrable as well, and hence it is relatively
weakly compact by the Dunford--Pettis theorem.
Passing to a subsequence, still denoted by $(u_k)$, we may therefore
assume that
\[
u_k\rightharpoonup u \quad\text{in }L^p(Y;\mu)
\]
for some $u\geq0$. Moreover, $u\neq0$. Indeed, if $u=0$, then,
by the compactness of $(G_\alpha^{V_j})^2$,
\[
(G_\alpha^{V_j})^2u_k\longrightarrow0\quad\text{in }L^p(Y;\mu).
\]
On the other hand,  by  \eqref{eq.uk-domination}, 
\[
0<r^2=r^2\|u_k\|\leq \|(G_\alpha^{V_j})^2u_k\|,
\]
a contradiction.
By \eqref{eq.order-resolvent-truncation}, 
\[
G_\alpha^{V}u_k\leq G_\alpha^{V_k}u_k=r_ku_k\leq G_\alpha^{V_n}u_k,\quad k\ge n\ge j.
\]
Passing to the weak limit gives
\[
G_\alpha^{V} u\leq\rho u\leq G_\alpha^{V_n}u,\quad n\ge j.
\]
By \eqref{eq.strong-resolvent-truncation}, we conclude that
\[
G_\alpha^{V} u=\rho u.
\]
As $u>0$, $\rho>0$, and $G_\alpha^{V}$ is positivity improving, we have $u\gg0$.
Assertion (ii) now implies that $\rho=r(G_\alpha^{V})$.
Finally, by (iii), we get the asserted convergence.
\end{proof}

\begin{lemma}
\label{th.pert1lm}
Assume that $V$ satisfies \eqref{eq.c1}.
\begin{enumerate}
\item[(i)] Suppose that $u\in L^p(Y;\mu)$ is nontrivial and
$G(|u|V^+)$ is finite $\mu$-a.e. Furthermore, suppose that
\begin{equation*}
u=-G(Vu).
\end{equation*}
Then
\begin{equation*}
u=\alpha G_\alpha^V u
\end{equation*}
for every $\alpha>s(A-V)$. Consequently,
\begin{equation*}
u\in\mathscr D(A-V),\qquad (A-V)u=0,
\end{equation*}
and hence $s(A-V)\ge0$.

\item[(ii)] Suppose that $u\in L^p(Y;\mu)$ is positive and nontrivial,
$G(uV^+)$ is finite $\mu$-a.e., and that there exists a supermedian
function $v$ such that
\begin{equation*}
u=-G(Vu)-v.
\end{equation*}
Then
\begin{equation*}
s(A-V)\ge0.
\end{equation*}
\end{enumerate}
\end{lemma}

\begin{proof}
Since $V^-\in\mathcal B_b(Y)$, the assumptions imply that
$G(|Vu|)$ is finite $\mu$-a.e. Set $g:=Vu$. By Lemma
\ref{lm.basic}(3), applied separately to $g^+$ and $g^-$, there exists,
for $\mu$-a.e. $x\in Y$, a uniformly integrable martingale $M^x$ such
that
\begin{equation*}
-Gg(X_t)=-Gg(x)+\int_0^t g(X_s)\,ds+M_t^x,\qquad t\ge0,\quad \mathbb P_x\text{-a.s.}
\end{equation*}

\smallskip
\noindent
\textit{Proof of (i).}
Since $u=-Gg$, we obtain
\begin{equation*}
u(X_t)=u(x)+\int_0^t V(X_s)u(X_s)\,ds+M_t^x.
\end{equation*}
Set
\begin{equation*}
Z_t:=\exp\left(-\int_0^tV(X_s)\,ds\right).
\end{equation*}
By the integration-by-parts formula
(see, e.g., \cite[Corollary II.6.2]{Protter}),
\begin{equation}
\label{eq.zzz}
Z_tu(X_t)=u(x)+\int_0^t Z_s\,dM_s^x.
\end{equation}
Since $V^-$ is bounded, $Z$ is bounded on every finite time interval.
Consequently, the stochastic integral on the right-hand side is a
martingale on every finite time interval. Indeed,
fix $T<\infty$, $c:=\|V^{-}\|_{\infty}$ and $H:=\int_0^{\infty}|V u|(X_s) \,ds$.
By the assumptions made, $\mathbb E_xH<\infty$ for $\mu$-a.e. $x\in E$.
Furthermore, 
\[
|Z_\sigma u(X_\sigma)| \leq e^{cT} G(|V u|)(X_\sigma)=e^{cT} \mathbb{E}_x[\int_\sigma^{\infty}
|V u|\left(X_s\right) \, d s | \mathcal{F}_\sigma]
 \leq e^{cT} \mathbb{E}_x\left[H | \mathcal{F}_\sigma\right]
\]
Consequently, taking expectations in \eqref{eq.zzz} gives
\begin{equation*}
u(x)=\mathbb E_x\left[e^{-\int_0^tV(X_s)\,ds}u(X_t)\right]=P_t^Vu(x)
\end{equation*}
for $\mu$-a.e. $x\in Y$ and every $t\ge0$. Hence
\begin{equation*}
P_t^Vu=u,\qquad t\ge0.
\end{equation*}
Since $(P_t^V)$ is a strongly continuous semigroup on $L^p(Y;\mu)$,
it follows that
\begin{equation*}
u\in\mathscr D(A-V),
\qquad
(A-V)u=0.
\end{equation*}
As $u$ is nontrivial, $0$ is an eigenvalue of $A-V$, and therefore
\begin{equation*}
s(A-V)\ge0.
\end{equation*}
Moreover, for every $\alpha>s(A-V)$,
\begin{equation*}
G_\alpha^Vu=\int_0^\infty e^{-\alpha t}P_t^Vu\,dt=\frac{1}{\alpha}u,
\end{equation*}
which proves (i).

\smallskip
\noindent
\textit{Proof of (ii).}
Let $\tilde v$ be the $\mathcal U$-excessive $\mu$-version of $v$
provided by Lemma \ref{lm.basic}(1). By Lemma \ref{lm.basic}(2) and
(4), there exist a positive additive functional $A$ and a local
martingale $N^x$ such that, for $\mu$-a.e. $x\in Y$,
\begin{equation*}
\tilde v(X_t)=\tilde v(x)-A_t+N_t^x,\qquad t\ge0,\quad \mathbb P_x\text{-a.s.}
\end{equation*}
Since
\begin{equation*}
u=-Gg-\tilde v,
\end{equation*}
we obtain
\begin{equation*}
u(X_t)=u(x)+\int_0^tV(X_s)u(X_s)\,ds+A_t+\widetilde M_t^x,\qquad t\ge0,
\end{equation*}
where $\widetilde M^x$ is a local martingale. Applying the
integration-by-parts formula gives
\begin{equation*}
Z_tu(X_t)=u(x)+\int_0^t Z_s\,dA_s+L_t^x,
\end{equation*}
where $L^x$ is a local martingale.
Let $(\tau_n)$ be a localizing sequence for $L^x$. Since $u$ is
positive and $A$ is increasing, we obtain
\begin{equation}
\label{eq.ineq915}
u(x)\le\mathbb E_x\left[Z_{t\wedge\tau_n}u(X_{t\wedge\tau_n})\right].
\end{equation}
Let $\sigma_n:=t\wedge\tau_n$. We claim that
$(Z_{\sigma_n}u(X_{\sigma_n}))_{n\ge1}$ is uniformly integrable
under $\mathbb P_x$ for $\mu$-a.e. $x\in Y$. Indeed, from
$u=-G(Vu)-\tilde v$ and $\tilde v\ge0$ we obtain $u\le \|V^-\|_\infty Gu$.
Moreover,
\begin{equation*}
Z_{\sigma_n}=\exp\left(-\int_0^{\sigma_n}V(X_s)\,ds\right)\le e^{t\|V^-\|_\infty}.
\end{equation*}
Thus it suffices to show that
$(Gu(X_{\sigma_n}))_{n\ge1}$ is uniformly integrable. By the strong
Markov property,
\begin{equation*}
Gu(X_{\sigma_n})=\mathbb E_x\left[\left.\int_{\sigma_n}^{\infty}u(X_s)\,ds\,\right|\,\mathcal F_{\sigma_n}\right]
\le
\mathbb E_x\left[\left.\int_0^{\infty}u(X_s)\,ds\,\right|\,\mathcal F_{\sigma_n}\right].
\end{equation*}
Since $u\in L^p(Y;\mu)$ and $G$ is bounded on $L^p(Y;\mu)$,
$Gu(x)<\infty$ for $\mu$-a.e. $x\in Y$. Hence
$\int_0^\infty u(X_s)\,ds$ is integrable under $\mathbb P_x$ for
$\mu$-a.e. $x\in Y$. The conditional expectations of a fixed
integrable random variable form a uniformly integrable family, and
therefore $(Z_{\sigma_n}u(X_{\sigma_n}))_{n\ge1}$ is uniformly
integrable.
Since $\tau_n\uparrow\infty$, we have $\sigma_n\uparrow t$ and, in
fact, $\sigma_n=t$ for all sufficiently large $n$, $\mathbb P_x$-a.s.
Consequently, letting $n\to\infty$ in \eqref{eq.ineq915} yields
\begin{equation*}
u(x)\le\mathbb E_x\left[e^{-\int_0^tV(X_s)\,ds}u(X_t)\right]=P_t^Vu(x)
\end{equation*}
for $\mu$-a.e. $x\in Y$ and every $t\ge0$.
Suppose, contrary to the assertion, that $s(A-V)<0$. Then, for every
$\alpha>0$ sufficiently large, integration against
$\alpha e^{-\alpha t}\,dt$ gives
\begin{equation*}
u\le\alpha G_\alpha^Vu.
\end{equation*}
Since $u$ is positive and nontrivial, it follows that $r(G_\alpha^V)\ge\frac{1}{\alpha}$.
On the other hand, Theorem \ref{th.pert1}(iii) gives
\begin{equation*}
r(G_\alpha^V)=\frac{1}{\alpha-s(A-V)}<\frac{1}{\alpha},
\end{equation*}
which is a contradiction. Hence $s(A-V)\ge0$.
\end{proof}

\begin{theorem}
\label{th.necessary}
Assume that \textup{(F1)--(F3)} hold and $a_0$ is bounded.
Suppose, moreover, that for $m$-a.e. $x\in E$, the function
\[
(0,\infty)\ni y\longmapsto h(x,y):=\frac{f(x,y)}{y}
\]
is non-increasing and on a set of positive measure the above mapping is strictly decreasing. 
If \eqref{eq1.1} admits a strictly positive
solution $u$, then
\[
\lambda_1(a_0)<0<\lambda_1(a_\infty).
\]
\end{theorem}

\begin{proof}
By the Doob transform introduced in Section \ref{sec7}, with the
function $\psi$ appearing in \textup{(F3)$_\psi$}, we may assume
without loss of generality that $(T_t)$ is sub-Markovian and
\textup{(F3)$_1$} holds. Indeed, this transform converts
\textup{(F3)$_\psi$} into \textup{(F3)$_1$} and preserves
the strict monotonicity of $y\mapsto f(x,y)/y$, as well as
the functions $a_0$, $a_\infty$ and the corresponding generalized
principal eigenvalues.
By Remark \ref{rem.main}, we may further assume
that $E$ is a Lusin topological space, $\mathcal A=\mathcal B(E)$, and
that there exists a right Markov process $\mathbb X$ representing
$(T_t)$.
Due to the assumptions made there exists a set $B\in\mathcal A$ such that $m(B)>0$, abd
\eqref{eq.slopes} exist and
\begin{equation}
\label{eq.slop1298}
a_\infty<h(\cdot,u)<a_0 \qquad m\text{-a.e.},\quad a_\infty<h(\cdot,u)<a_0 \qquad m\text{-a.e. on }B.
\end{equation}
In particular,
\[
h^+(\cdot,u)\in L^\infty(E;m),
\]
which gives  the first part of \eqref{eq.c1} for $V:=-h(\cdot,u)$.
To verify the integrability condition in \eqref{eq.c1}, fix $\delta>0$.
By \textup{(F3)$_1$}, for any  $y>0$,
\[
h^-(x,y)\le C_\delta+\delta^{-1}f^-(x,y).
\]
Since $Rf^-(\cdot,u)\in L^p(E;m)$, 
\[
\mathbb E_x\int_0^\zeta f^-(X_s,u(X_s))\,ds<\infty
\]
for $m$-a.e. $x$. Consequently,
\[
\int_0^t h^-(X_s,u(X_s))\,ds<\infty\quad\mathbb P_x\text{-a.s. on }\{t<\zeta\},
\]
which proves the required part of \eqref{eq.c1}.
Since $u$ is strictly positive, Lemma~\ref{th.pert1lm}(i) and Theorem \ref{th.pert1}, applied
with $A=L$ and $V=-h(\cdot,u)$, yield
\[
s\bigl(L+h(\cdot,u)\bigr)=0.
\]
By \eqref{eq.slop1298} and Theorem~\ref{th.pert1}
\[
0=s\bigl(L+h(\cdot,u)\bigr)<s(L+a_0)=-\lambda_1(a_0).
\]
It remains to prove that $\lambda_1(a_\infty)>0$.
By \eqref{eq.slop1298} and since  $h(\cdot,u)$ is finite a.e., 
there exist $\varepsilon>0$, 
and a measurable set $B_1\subset B$ with $m(B_1)>0$ such that
\[
h(\cdot,u)-a_\infty\ge 2\varepsilon \qquad\text{on }B_1.
\]
Set $h(\cdot,u)_\varepsilon:=h(\cdot,u)-\varepsilon\mathbf1_{B_1}$.
Then
\[
h(\cdot,u)_\varepsilon\le h(\cdot,u)
\]
with strict inequality on a set of positive measure. Moreover,
$-h(\cdot,u)_\varepsilon$ satisfies \eqref{eq.c1}, since it differs from
$-h(\cdot,u)$ by the bounded function $\varepsilon\mathbf1_B$.
By Theorem~\ref{th.pert1}(v),
\[
s\bigl(L+h(\cdot,u)_\varepsilon\vee(-k)\bigr)\downarrow s\bigl(L+h(\cdot,u)_\varepsilon\bigr)<s\bigl(L+h(\cdot,u)\bigr)=0.
\]
Hence, for all sufficiently large $k$,
\[
-\lambda_1(a_\infty^k)=s(L+a_\infty^k)\le s\bigl(L+h(\cdot,u)_\varepsilon\vee(-k)\bigr)<0.
\]
Consequently,
\[
\lambda_1(a_\infty)=\sup_{k\ge1}\lambda_1(a_\infty^k)>0.
\]
\end{proof}

\begin{theorem}
Suppose that $a_0$ is bounded. Assume that  for $m$-a.e. $x\in E$,
\[
(0,\infty)\ni y\longmapsto h(x,y):=\frac{f(x,y)}{y}
\]
is  non-increasing and on a set of positive measure  the above mapping is decreasing.
Then \eqref{eq1.1} admits at most one strictly
positive solution.
\end{theorem}

\begin{proof}
By the Doob transform used in the proof of Theorem \ref{th.main}, it
suffices to consider the case where $(T_t)$ is sub-Markovian. 
By Remark \ref{rem.main}, we may further assume that $E$ is a Lusin
topological space, $\mathcal A=\mathcal B(E)$, and that there exists a right
Markov process $\mathbb X$ representing $(T_t)$.
Let $u_1,u_2$ be strictly positive solutions to \eqref{eq1.1} and set
\[
u:=u_1\vee u_2.
\]
Since every solution is, in particular, a subsolution, Proposition
\ref{prop.maxsub} shows that $u$ is a subsolution. Thus, for some
supermedian function $v$,
\[
u=Rf(\cdot,u)-v.
\]
Since $f(\cdot,u)=h(\cdot,u)u$, this can be written as
\[
u=R\bigl(h(\cdot,u)u\bigr)-v.
\]
The potential $V:=-h(\cdot,u)$ satisfies \eqref{eq.c1}
(see the reasoning in the proof of Theorem \ref{th.necessary}).
Hence Lemma \ref{th.pert1lm}(ii) yields
\begin{equation}
\label{eq.in167}
s\bigl(L+h(\cdot,u)\bigr)\ge0.
\end{equation}
On the other hand, since $u_i$ is a strictly positive solution,
Lemma \ref{th.pert1lm} and Theorem \ref{th.pert1} give
\[
s\bigl(L+h(\cdot,u_i)\bigr)=0,
\qquad i=1,2.
\]
Since $h(x,\cdot)$ is non-increasing, we have
\[
h(\cdot,u)\le h(\cdot,u_i),\qquad i=1,2.
\]
Therefore, by Theorem~\ref{th.pert1}(iv),
\[
0\le s\bigl(L+h(\cdot,u)\bigr)
\le s\bigl(L+h(\cdot,u_i)\bigr)=0.
\]
The strict comparison assertion in the same theorem implies
\[
h(\cdot,u)=h(\cdot,u_1)=h(\cdot,u_2)
\qquad m\text{-a.e.}
\]
Set $V:=h(\cdot,u)$. By Lemma~\ref{th.pert1lm}(i),
both $u_1$ and $u_2$ belong to $\mathscr D(L+V)$ and satisfy
\[
(L+V)u_i=0,\qquad i=1,2.
\]
The resolvent of $L+V$ is positivity improving and its square
is compact by Theorem~\ref{th.pert1}(i). Hence $u_2=c\,u_1$ for some $c>0$.
Since
\[
h(x,u_1(x))=h(x,c\,u_1(x))
\qquad\text{for }m\text{-a.e. }x\in B,
\]
the strict monotonicity of $h(x,\cdot)$ on $B$, together with
$m(B)>0$ and $u_1>0$ $m$-a.e., yields $c=1$.
Thus $u_1=u_2$ $m$-a.e.
\end{proof}

\end{document}